\documentclass[usenames,dvipsnames,12pt,english]{amsart}

\usepackage{tcolorbox}
\usepackage{xcolor}
\usepackage{extpfeil}
\usepackage{todonotes}
\usepackage[unicode]{hyperref}
\usepackage{wasysym}
\allowdisplaybreaks

\usepackage{pdfsync}
\usepackage{geometry}
\usepackage{mathscinet}

\usepackage{lscape}
\usepackage{graphicx}
\usepackage{epstopdf}
\usepackage[all,cmtip]{xy}

\usepackage{amsmath,amsthm,amssymb,enumerate}
\usepackage{latexsym}
\usepackage{amscd}
\usepackage{amssymb}
\usepackage{enumitem}
\usepackage{mathtools}
\numberwithin{equation}{section}

\theoremstyle{plain}  % These will have emph text
\newtheorem{thm}[equation]{Theorem}
\newtheorem{prop}[equation]{Proposition}

\newtheorem{lemma}[equation]{Lemma}

\theoremstyle{definition}  % These will have plain text
\newtheorem{defin}[equation]{Definition}

\newtheorem{remark}[equation]{Remark}

\newenvironment{numequation}
{
\begin{equation}}
{\end{equation}
}%%%%\vspace{-.17in}}
\newcommand{\SESM}{Eilenberg-Mac\,Lane}

\newcommand{\ra}{\rightarrow}

\newcommand{\lra}{\longrightarrow}

\newcommand{\crr}[1]{{\color{black}#1}}

\newcommand{\crg}[1]{{\color{black}#1}}

\newcommand{\crb}[1]{ {\color{black} #1 }}

\newcommand{\crbr}[1]{ { \color{black} #1 }}

\newcommand{\aand}[0]{\mbox{and}\qquad  }
\newcommand{\aaand}[0]{\mbox{and}\quad  }

\newcommand{\Z}{\mathbf Z}
\newcommand{\C}{\mathbb C}

\newcommand{\cp}{\C P^\infty}

\renewcommand{\SS}{spectral sequence}
\newcommand{\SSS}{Serre spectral sequence}
\newcommand{\ASS}{Adams spectral sequence}
\newcommand{\AHSS}{Atiyah-Hirzebruch spectral sequence}
\newcommand{\UCSS}{universal coefficient spectral sequence}
\def\zt{{\mathbf Z}/2}
\def\zp{{\mathbf Z}/p}

\def\ot{\otimes}
\def\Ebar{\overline{E}}

\def\E{E ( Q_n)}
\def\k{k(n)}

\DeclareMathOperator{\Ext}{Ext}

\newcommand{\mycases}[1]{\left\{\begin{array}{ll}#1\end{array}\right.}

\allowdisplaybreaks

\begin{document}

\thanks{We thank the referee for two extraordinarily careful readings of
this paper, with a complete rewrite in between.  This is a much better paper
because of the referee.  Thanks.}

%\pdfpagewidth=8.5in \pdfpageheight=11in

\today

\title[The connective {Morava} $K$-theory of $K(\zp,2)$] {The
connective {Morava} $K$-theory of the second mod $p$
{Eilenberg-MacLane space} }

\author{Donald M. Davis} \address{Department of Mathematics, Lehigh
University\\Bethlehem, PA 18015, USA} \email{dmd1@lehigh.edu}

\author{Douglas C. Ravenel} \address{Department of Mathematics,
University of Rochester \\ Rochester, NY 14627, USA}
\email{dravenel@ur.rochester.edu}

\author{W. Stephen Wilson} \address{Department of Mathematics, Johns
Hopkins University\\Baltimore, MD 01220, USA} \email{wwilson3@jhu.edu}

\begin{abstract}
We develop tools for computing the connective n-th Morava K-theory of
spaces.  Starting with a Universal Coefficient Theorem that computes
the cohomology version from the homology version, we show that every
step in the process of computing one is mirrored in the other and that
this can be used to make computations.  As our example, we compute the
connective $n$-th Morava $K$-theory of the second mod $p$ Eilenberg-MacLane
space.
\end{abstract}



\maketitle

\tableofcontents
\section{Introduction}

Being able to compute is central to much of algebraic topology.
Computing generalized (co)homology theories of basic spaces usually
runs from difficult to impossible.  One exception has been the
extraordinary $K$-theories of Jack Morava, $K(n)_{*}(X)$.  They have a
K\"unneth isomorphism that makes them more tractable to computations
than most.

There is a connective version of Morava $K$-theories, $k(n)_{*}(X)$, and in this
paper we make some progress towards computing with this.  In particular, we
develop some tools that can be applied to this problem in general, and then
we apply them to compute the $n$th connective Morava $K$-theory of the second
mod $p$ Eilenberg-MacLane space, $K_{2} = K(\zp,2)$, where $\zp$ is the integers
modulo the prime $p$.

% The origin of this problem is now securely in the past.
% It was known from
Anderson-Hodgkin \cite{AndHod} showed that $K(1)_{*}(K_{2})$ was
trivial.  The third author searched, periodically over the decades,
for the differentials in the Atiyah-Hirzebruch spectral sequence that
would reduce the already small group\linebreak $H_{*}(K_{2};K(1)_{*})$
to zero at $p=2$.  The differentials in the {\AHSS} are the same as
those in the Adams spectral sequence, so this paper finally gives the
third author great satisfaction.  The project grew into this paper.

The main computation of the paper is to compute both $k(n)_{*}(K_{2})$ and
$k(n)^{*}(K_{2})$ as $k(n)_{*}$ (and $k(n)^{*}$) modules.  The $n=1$ case is
essential for the first and third authors' determination of $ku^{*}(K_{2})$ and
$ku_{*}(K_{2})$ for all primes in \cite{DW}.  
It was also useful in the first author's 
determination of $ko^*(K_2)$ and $ko_*(K_2)$ in \cite{D},
\crg{with the $E_2$ page of the Adams spectral sequence 
for this computed in \cite{Davis-Ext}.}

One of our main tools is obtained by combining results of Robinson and Lazarev
for computational purposes.

\begin{thm}
\label{uct}
For $X$ a space of finite type with $K(n)_{*}(X)$ finitely generated over $K(n)_{*}$,
there is a universal coefficient spectral sequence that collapses:
$$
\Ext_{k(n)_{*}}^{s,t}(k(n)_{*}(X),k(n)_{*}) \Rightarrow k(n)^{s+t}(X)
$$
\end{thm}

In \cite{Rob-UCT} Alan Robinson created the {\UCSS} for homology
theories satisfying certain hypotheses. These were shown to be
satisfied by $k (n)_{*}$ by him in \cite{Rob-kn} and later by Andrey
Lazarev in \cite{Laz-kn}.  We will show the {\UCSS } collapses in this case.

From this result, we derive the next important tool.

\begin{thm}[{\bf The Pairing}] \label{pairing}
For $X$ a space of finite type
with $K(n)_{*}(X)$ finitely generated over $K(n)_{*}$, there is a
differential $d^r(\alpha ) = v^r \beta $ in the {\ASS} for $k(n)^{*}(X)$
if and only if there is a corresponding $d_r(\beta ') = v^r \alpha '$
in the {\ASS} for $k(n)_{*}(X)$, with $|\alpha |=|\alpha '|$ and $|\beta
| = |\beta '|$.
\end{thm}

It is the interaction between $k(n)$ cohomology and homology from
these two results that allows us to do our computation.  Theorem
\ref{uct} gives a duality of sorts between $k(n)_{*}(X)$ and
$k(n)^{*}(X)$, but Theorem \ref{pairing} goes even further and says that
there is a duality every step of the way in the computation.  In our
case, we have that $K_{2}$ is an H-space so both Adams spectral
sequences are multiplicative.  Although this does not give us a Hopf
algebra, there is enough similarity in the structure that we can make
good use of it.

The plan of the paper is to state the results of the main computation
in the next section.  We set up some notation in Section
\ref{notation}.  In Section \ref{qhomsec} we compute the $E_{2}$ term of
the {\ASS} for $k(n)^{*}(K_{2})$.  In Section \ref{sec-illustration} 
we illustrate its behavior for $n=2$.  We give some necessary definitions
and numbers in Section \ref{numdef}. In Section \ref{prelim}, we prove
the two theorems in the introduction and establish some other
preliminaries we need.  All the hard work is done in Section
\ref{thess} where the differentials are computed.  The results for
$k(n)_{*}(K_{2})$ are all collected in Section \ref{alldual} and the final
section is devoted to describing the results at $p=2$.

\section{Statement of results}

In this section we define only what we need to efficiently
state the results of our main computation of $k(n)^{*}(K_{2})$.
Many details will be properly developed later.

\subsection{ Basic notation.} All our cohomology and homology groups will be
mod $p$.  The connective $n$th Morava $K$-theory spectrum, $k(n)$, has
$k(n)^{*} = \zp[v_{n}]$ with\linebreak $| v_{n} | =\crr{-}2(p^{n} -1)$.

% \todo[inline,color=green]{6/17/24  The use of $v_{n}$ rather than $v$ is
% ubiquitous, so I am going with it.}%\bigskip

We let $P(x)$, $E(x)$, and $\Gamma(x)$ be the polynomial, exterior,
and divided power algebras on $x$ (which could be a single generator
or a set of generators) over $\zp$.  In addition, we need the
truncated polynomial algebra, $T_k(x) := P(x)/(x^k)$, and its dual,
$\Gamma_k(x)$.

The divided power algebra $\Gamma (x)$ for a single $x$ is additively
generated by elements $\gamma_{i} (x)$ for $i\geq 0$ (the divided
powers of $x$) with $|\gamma_{i} (x) |=i|x|$ and
\begin{displaymath}
\gamma_{i} (x)\gamma_{j} (x)=\binom{i+j}{i}\gamma_{i+j}(x).
\end{displaymath}
As an  algebra, $\Gamma(x)$ is $
T_p(\gamma_{p^m}(x):m\geq 0)$.  For $p=2$, this is an exterior
algebra.

For a rational number $x$, $\lfloor x \rfloor$, the
{\em floor of $x$}, denotes the largest integer not exceeding $x$, and
$\lceil x \rceil$, the {\em ceiling of $x$}, denotes the smallest
integer not exceeded by $x$.

For a graded connected $\Z/p$-algebra $A$, we let $\overline{A}$
denote the augmentation ideal of $A$, its \crr{vector} space of
positive degree elements.

\bigskip

\subsection{ The mod $p$ cohomology of $K_{2}:=K (\Z/p,2)$. }In what follows, {\em all tensor products are over $\Z/p$ unless
otherwise stated}.

To compute with the {\ASS}, we need (for $p$ an odd prime)
\begin{align}
\begin{split}
\label{coho}
H^{*}K_{2} &=
P(\iota_{2})\otimes  P(z_i:i>0 ) \otimes  E(u_i:i\geq 0 ) \\
\mbox{with}\qquad
|\iota_{2}|& = 2,\qquad
| z_i | =
 2p^{i} +2 \qquad\mbox{and}\qquad   | u_i | =  2p^{i} + 1.
% H^{*}K_{2} &=
% P(\iota_{2}) \underset{i > 0}{\ot} P(z_i ) \underset{i \ge 0}{\ot} E(u_i ) \\
% | z_i |& =
%  2(p^{i} +1) \qquad | u_i | =  2p^{i} + 1
\end{split}
\end{align}
Let $y_{n,j} = \iota_{2}^{p^{j}}$.  In particular, $\iota_{2} = y_{n,0}$
and $\iota_{2}^{p} = y_{n,0}^{p} = y_{n,1}$.  In general, $y_{n,j}^{p} = y_{n,j+1}$
with $| y_{n,j} | = 2p^{j}$.

For $p=2$, $H^{*}K_{2}$ has a similar description but with
$u_{i}^{2}=z_{i+1}$.  We will say more about this in Section
\ref{mod2}.

We define $w_{n,i}\in H^{*}K_{2}$ for $i\geq 0$ by
% Defined in ~/math/diary/diary.sty
\begin{equation}\label{eq-w}
\begin{split}
w_{n,i}:=\mycases{
u_{n}
       &\mbox{for }i=0\\
u_{n+i} - u_{n-i} z_{i}^{p^{n} - p^{n-i}}
       &\mbox{for }0< i\leq n\\
 y_{n,i-n-1/2}z_{i}^{p^{n}-1}\\
 \qquad =w_{n,i-n-1}y_{n,i-n-1}^{p-1}z_{i}^{p^{n}-1}
       &\mbox{for }i\geq n+1,
}
\end{split}
\end{equation}%

\noindent where
\begin{equation}
\label{zhalf}
% w_{n,i+1/2} : = y_i^{p-1} w_{n,i} \qquad \mbox{for integers $i\geq 0$} .
y_{n,i+1/2} : = y_{n,i}^{p-1} w_{n,i} \qquad \mbox{for integers $i\geq 0$} .
\end{equation}

\crg{
In general, all our variables, such as $n,i,j,k,s$, are non-negative 
integers.  The number $1/2$ arises often, and should be clear from context.}

In Section \ref{notation} we will see that there is an {\ASS}
converging to $k (n)^{*}K_{2}$ for which the input is $k
(n)^{*}\otimes H^{*} K_{2}$.  It is indexed in such a way
that
\begin{itemize}
\item [$\bullet$] the filtration and dimension of $v_{n}$ are 1 and
$-2 (p^{n}-1)$,

\item [$\bullet$] the elements of $H^{*}K_{2}$ have their usual
positive degrees and Adams filtration 0,  and

\item [$\bullet$] differentials $d^{r}$ raise rather than lower degree
by 1, while \crr{raising} filtration by $r$.
\end{itemize}

In Section \ref{qhomsec} we will see that the action of the
Milnor primitive $Q_{n}$ on $H^{*}K_{2}$ \crr{gives us $d^1(x)=v_n Q_n(x)$.
From this we get}
% Defined in ~/math/diary/diary.sty
\begin{numequation}\label{eq-d1ui}
\begin{split}
d^{1} (y_{n,0})
 & = v_{n}u_{n}\\
\aand d^{1} (u_{s})
 & = \mycases{
v_{n}z_{n-s}^{p^{s}}
       &\mbox{for }0\leq s<n\\
0      &\mbox{for }s=n\\
v_{n}z_{s-n}^{p^{n}}
       &\mbox{for }s>n.}
\end{split}
\end{numequation}%

\noindent This implies that for $w_{n,s}$ as in (\ref{eq-w}), $d^{1}
(w_{n,s}) = 0$ for $0\leq s\leq n$.  We can regard $w_{n,s}$ for such
$s$ as a substitute for $u_{n+s}$ \crb{that survives to $E_{2}$}.

In Section \ref{thess} for $p$ odd and Section \ref{mod2} for $p=2$,
we will see that there are higher Adams differentials
% Defined in ~/math/diary/diary.sty
\begin{numequation}\label{eq-higher}
\begin{split}
d^{\rho_{n}(i)}(y_{n,i})
 & = v_{n}^{\rho_{n}(i)}w_{n,i}\\
\aand
d^{\rho_{n}(i+1/2)}( y_{n,i+1/2})
 & = v_{n}^{\rho_{n}(i+1/2)} z_{n+i+1}
\end{split}
\end{numequation}%

\noindent for integers $i\geq 0$, where the numbers $\rho_{n} (i)$ and
$\rho_{n} (i+1/2)$ are given in Lemma \ref{lem-ab}.  The latter are
uniquely determined by the dimensions of the elements in question. For
integers $0\leq i\leq n$, they are
\begin{displaymath}
\rho_{n} (i)=p^{i}
\qquad \aand
\rho_{n} (i+1/2)= (p-1) p^{i}.
\end{displaymath}

\noindent
In particular $\rho _{n} (0)=1$, so the
first differential of (\ref{eq-higher}) for $i=0$ coincides with the
first differential of (\ref{eq-d1ui}).

\subsection{ The effect of the Adams $d^{1}$.}\label{subsec-d1}

  {\em The additional $d^{1}$s of
(\ref{eq-d1ui}) make the passage from $E_{1}$ to $E_{2}$ more
complicated than the passage to higher terms brought about by the
higher differentials of (\ref{eq-higher}).} We will outline these
processes here in order to motivate the complicated expressions in
Theorem \ref{knkn}, our main computational result.

In order to work out the
implications of (\ref{eq-d1ui}), the following additive isomorphisms
and definitions for each positive $n$ and $i\geq 0$ are convenient.
% Defined in ~/math/diary/diary.sty
\begin{numequation}\label{eq-convenient}
\begin{split}
P (y_{n,i})
 & \phantom{:}\cong T_{p} (y_{n,i})\otimes P (y_{n,i+1}),   \\
E (u_{s}:s\geq 0)
 & \phantom{:}\cong EE_{n}\otimes E (w_{n,0})\otimes  W_{n,0},
   \mbox{ where}   \\
EE_{n}
 & := E (u_{s}:0\leq s<n)\otimes E (u_{\crr{2}n+s}:s>0)
\qquad \mbox{ and }\\
W_{n,i}
 & :=E (w_{n,i+s}:1\leq s\leq n)
     \quad \mbox{for  $w_{n,i+s}$ as in (\ref{eq-w}),} \\
P (z_{s}:s> 0)
 & \phantom{:}\cong L_{n}\otimes TZ_{n,0}\otimes PZ_{n}, \mbox{ where}   \\
L_{n}
 &: = \bigotimes_{0<s <n}T_{p^{n-s }} (z_{s }),\\
TZ_{n,i}
 & := T_{p^{n}} (z_{n+s}:s>i),\qquad \mbox{  and }\\
PZ_{n} &:= P (z_{s}^{e_{n} (s)}:s>0)\qquad \mbox{with }
 e_{n} (s)
 :=\mycases{
p^{n-s}    &\mbox{for }0 < s \leq n\\
p^{n}      &\mbox{for }s>n.
}
\end{split}
\end{numequation}%

\noindent We will make use of these $W_{n,i}$ and $TZ_{n,i}$ for $i>0$
later.

\crr{
\begin{remark}
Although stated as additive isomorphisms, much of the algebra structure is
preserved and we need it.  For example, the additive isomorphism for
$P(y_{n,i})$ comes from the multiplicative extension
\[
%\Z/p \lra
P(y_{n,i+1}) \lra P(y_{n,i}) \lra T_p(y_{n,i}), 
%\lra \Z/p,
\]
\end{remark}  }

\noindent \crb{meaning that $P(y_{n,i})$ is a free module over the subring
$P(y_{n,i+1})$ and
\begin{displaymath}
T_p(y_{n,i}) = P(y_{n,i})\otimes_{P(y_{n,i+1})}\Z/p.
\end{displaymath}

\noindent

} Here, if we have $d^r(y_{n,i}) \neq 0$, we have
$d^r(y_{n,i}^p=y_{n,i+1}) = 0$, but this requires the multiplicative
structure.  Similarly, we have
\[
%0 \lra
PZ_n \lra P(z_s: s > 0) \lra L_n \otimes TZ_{n,0}.%}
% \lra 0.
\]
%}  EDIT

Putting this all together, we have
% Defined in ~/math/diary/diary.sty
\begin{numequation}\label{eq-all}
\begin{split}
H^{*}K_{2}
 & \cong   T_{p} (y_{n,0})\otimes P (y_{n,1}) \otimes EE_{n}
  \otimes E (w_{n,0})\otimes W_{n,0}
   \otimes L_{n}\otimes TZ_{n,0} \otimes PZ_{n}.
\end{split}
\end{numequation}%

\crr{ We see in (\ref{eq-bigraded2}) that the $d^{1}$ of
(\ref{eq-d1ui}) on $k(n)^*$ tensored with (\ref{eq-all}) is confined
to $k(n)^* \otimes D_{1}$ with
\begin{numequation}\label{onlyd1}
D_{1}\coloneqq  T_{p} (y_{n,0})
  \otimes E(w_{n,0})
\otimes EE_{n}
    \otimes PZ_{n}.
\end{numequation}%

\noindent
The $d^1$ homology of $k(n)^*\otimes T_p(y_{n,0}) \otimes E(w_{n,0})$ is just
$k(n)^* \otimes E(y_{n,1/2}) \oplus T_{p-1}(y_{n,0}) \otimes \overline{E(w_{n,0})}$.
This is illustrated in the top diagram of
(\ref{eq-diagrams}) for $j=0$, where $\rho_{n} (0)=1$.
The $d^1$ homology of $k(n)^* \otimes EE_n \otimes PZ_n$ is
$k(n)^* $ plus elementary $v_n$ torsion elements.  Combining these results is tricky so
we just observe that
the elementary $v_n$-torsion in (\ref{onlyd1}) is the image of $Q_n$, i.e. $Q_n(D_{1})$.
If we remove the elements in $E(y_{n,1/2})$ from (\ref{onlyd1}), what remains is free over
$E(Q_n)$.

It is straightforward to compute
the Poincar\'e series from this information.
Although this does not give an explicit base,
it isn't hard to filter it and get a basis for the associated graded version.
For example, for $EE_n\otimes PZ_n$, we would have
\begin{numequation}\label{assoc-graded}
\begin{split}
\bigoplus_{0 <  k \le n}
   & E(u_i:0 \le i < n-k)\otimes E(u_i: 2n < i )
    \otimes \overline{P(z_k^{e_n(k)})} \otimes P(z_i^{e_n(i)}: i>k) \\
\crb{\oplus }
\bigoplus_{2n < k }
   & E(u_i: i > k) \otimes \overline{P(z_{k-n}^{p^n})}
             \otimes P(z_i^{p^n}: i>k-n)
\end{split}
\end{numequation}%
}

\noindent \crr{We} define
% Defined in ~/math/diary/diary.sty
\begin{numequation}\label{eq-SM0}
\begin{split}
S_{n,0}
 &:= \crr{ Q_n(D_{1})} \\
M_{n,i}
 &: = P (y_{n,i+1})\otimes W_{n,i}\otimes L_{n}\otimes TZ_{n,i}
\qquad \mbox{for $i\geq 0$}.\\
M_{n,i+1/2}
 & :=P (y_{n,i+1})\otimes W_{n,i}
           \otimes L_{n}\otimes  TZ_{n,i+1}\quad \mbox{for }i\geq 0 ,\\
S_{n,i}
 &:= T_{\rho_{n}(i)}(v_{n})\ot T_{p-1}(y_{n,i})
      \ot \overline{E(w_{n,i})} \\
 & \phantom{:} = T_{\rho_{n}(i)}(v_{n})\ot
\left\{w_{n,i}y_{n,i}^{s}:0\leq s\leq p-2 \right\}
  \quad \mbox{for }i>0 ,\\
\aand
S_{n,i+1/2}
 &:= T_{\rho_{n}(i+1/2)}(v_{n})\ot  \overline{T_{p^{n}}(z_{n+i+1})} \\
 &\phantom{:} = T_{\rho_{n}(i+1/2)}(v_{n})\ot
 \left\{z_{n+i+1}^{s}:1\leq s< p^{n} \right\} \quad \mbox{for }i\geq 0.
\end{split}
\end{numequation}%

\noindent {\em These will figure in Theorem \ref{knkn}.}
\crbr{Using this notation, we have computed the elementary $v_n$-torsion
as $S_{n,0}\otimes M_{n,0}$.}

\subsection{ The effect of higher Adams differentials.}  The higher
differentials of (\ref{eq-higher}) are easier to deal with since each
is nonzero on a single multiplicative generator.  They are illustrated
in the diagrams of (\ref{eq-diagrams}) below.

The vector spaces $S_{n,i}$ and $S_{n,i+1/2}$ of (\ref{eq-SM0}) can
also be written as
\begin{align*}
S_{n,i}
 & = d^{\rho_{n} (i)}\left(\overline{T_{p} (y_{n,i})} \right)/
                    v_{n}^{\rho_{n} (i)}   \\
\aand
S_{n,i+1/2}
 & = d^{\rho_{n} (i+1/2)}\left(\overline{E (y_{n,i+1/2})}\crr{\otimes T_{p^n-1}(z_{n+i+1})} \right)
           /v_{n}^{\rho_{n} (i+1/2)}.
\end{align*}

We can now state the main computational result of this paper.

\begin{thm}\label{knkn}
For an odd prime $p$,   $\k^{*}(K_{2})$
has the following three summands as a $k(n)^{*}$-module:
\begin{enumerate}[label={(\roman*)},itemindent=1em]

\item \label{knkni} The {\em $k (n)^{*}$ free summand},
$k(n)^{*}\otimes L_{n}$, for $L_{n}$ as in (\ref{eq-convenient}).

\item \label{knknii}  The {\em higher torsion summand},
\begin{displaymath}
 \bigoplus_{\ell >0}\left( M_{n,\ell /2}\otimes  S_{n,\ell /2} \right),
\end{displaymath}

\noindent for $M_{n,\ell /2}$ and $S_{n,\ell /2}$ as in (\ref{eq-SM0}).

\item \label{knkniii} The {\em elementary torsion summand},
$S_{n,0}\otimes M_{n,0}$ as in (\ref{eq-SM0}).

\end{enumerate}
\end{thm}

\begin{remark}
{\bf The $v_{n}$-torsion free summand.}
Inverting $v_{n}$ kills all but the first summand of $k (n)^{*}K_{2}$,
which becomes $K(n)^{*}(K_{2})$, as described in \cite[dual to Theorem
11.1]{RW:CF}.  This $k (n)^{*}$-free summand is all that appears in negative
degrees, where it is finite in each degree.  In addition, every
positive degree is also finite.
\end{remark}

% \noindent a tensor factor of the elementary torsion summand.  However
% it is not a subgroup of it since it gets tensored with $d^{1}
% (\overline{EE}) /v_{n}$, which does not contain the unit class. Thus
% the elements $z_{n+i}$ for $i>0$ are not killed by $v_{n}$, but only
% by $v_{n}^{\rho_{n} (i-1/2)}$.
% \end{remark}

\begin{remark}\label{rem-1}
{\bf The multiplicative structure.} Theorem \ref{knkn} describes a
ring as well as a $k (n)^{*}$-module, but we can only show that the
ring structure is that of the Adams $E_{\infty }$-term. We cannot rule
out nontrivial multiplicative extensions.  For $n>2$, we cannot show
by dimensional arguments that $v_{n}z_{n+i}^{p^{n}}=0$ for $i>0$.
Let
\begin{displaymath}
\kappa_{n} = \prod_{0<i<n}z_{n-i}^{p^{i}-1},
\end{displaymath}

\noindent the top class in $L_{n}$.  We can show by induction on $n$
that $|\kappa_{n}|=2 (p^{n}-1) (n-1)$ using the identity
\begin{displaymath}
\kappa_{n+1}=\kappa_{n}^{p} (z_{1}z_{2}\cdots z_{n})^{p-1}.
\end{displaymath}

\noindent We cannot rule out the
multiplicative extension
\begin{displaymath}
v_{n}z_{n+i}^{p^{n}}
  = v_{n}^{n-1}\kappa_{n}z_{2n+i}
\end{displaymath}

\noindent (note that $|v_{n}z_{n+i}^{p^{n}}|=|z_{2n+i}|$ and
$|v_{n}^{n-1}\kappa_{n}|=0$) for $n>2$ and $i>0$.
\end{remark}

\section{Our Adams spectral sequence notation}
\label{notation}

The $k(n)$ under consideration here is the the connective version of
Morava's $n$th extraordinary $K$-theory $K (n)$.  We have
\begin{align*}
\pi_{*}k(n)& = P(v_{n}) && \mbox{with }|v_{n}| = 2(p^{n}-1)    \\
%k(n)^{*}& = P(v_{n}) && \mbox{with }|v_{n}| = -2(p^{n}-1)    \\
% \mbox{and}\qquad
% k(n)_{*}& = P(v_{n}) && \mbox{with }|v_{n}| = 2(p^{n}-1) \\
% \mbox{with}\qquad
\aand H^{*}(k(n))& = \mathcal{A}/\mathcal{A} (Q_n),
\end{align*}
where $\mathcal{A}$ is
the mod $p$ Steenrod algebra and $Q_n$ is the $n$th Milnor primitive.

% \todo[inline,color=green]{7/2/24. In this treatment, $v_{n}\in
% E_{2}^{1,2p^{n}-1}$ in both the cohomological and homological cases.
% In the notation of (\ref{eq-cohom-case}) and (\ref{eq-hom-case}) we
% have $v_{n}\in G^{2}_{2-2p^{n},1}$ in cohomology and $v_{n}\in
% G_{2}^{2p^{n}-2,1}$ in homology.}\bigskip

We have
% Defined in ~/math/diary/diary.sty
\begin{numequation}\label{eq-kn-groups}
\begin{split}
k (n)^{i}X:=[X, k (n)]_{-i}=\pi_{-i}F[X,k (n)]
\qquad \aand
k (n)_{i}X:=\pi_{i} (k (n)\wedge X).
\end{split}
\end{numequation}%

\noindent Given $\alpha \in k (n)_{i}X$ represented by a map $S^{i}\to
k (n)\wedge X$, and $\beta \in k (n)^{j}X$ represented by a map $X\to
\Sigma^{j}k (n)$, we get an element $\langle \alpha ,\,\beta
\rangle\in \pi_{i-j}k (n)$, which is the composite
\begin{displaymath}
\xymatrix
@R=4mm
@C=15mm
{
{S^{i}}\ar[r]^(.4){\alpha }
  &{k (n)\wedge X}\ar[r]^(.42){k (n)\wedge \beta }
    &{k (n)\wedge \Sigma^{j}k (n)}\ar[r]^(.58){m}
      &{\Sigma^{j}k (n),}
}
\end{displaymath}

\noindent where $m:k (n)\wedge k (n)\to k (n)$ is the multiplication in the
ring spectrum $k (n)$.

The groups of (\ref{eq-kn-groups}) can be computed with the {\ASS} as
follows.  In the above, $k (n)$ is the spectrum representing
connective Morava K-theory and $F (X,Y)$ denotes the function spectrum
for maps of spectra $X\to Y$.  The ring structure on $k (n)$ allows us
to extend the map $S^{|v_{n}|}\to k (n)$ representing $v_{n}$ to the
map $v_{n}$ in the fiber sequence
% Defined in ~/math/diary/diary.sty
\begin{numequation}\label{eq-fiber-seq}
\begin{split}
\xymatrix
@R=4mm
@C=10mm
{
{\Sigma^{|v_{n}|}k (n)}\ar[r]^(.6){v_{n}}
  &{k (n)}\ar[r]^(.5){j}
    &{H/p}\ar[r]^(.38){\delta }
      &{\Sigma^{2p^{n}-1}k (n)}
}
\end{split}
\end{numequation}%

\noindent where $H/p$ is the mod $p$ {\SESM} spectrum.  The composite
% Defined in ~/math/diary/diary.sty
\begin{numequation}\label{eq-Qn}
\begin{split}
\xymatrix
@R=4mm
@C=10mm
{
{H/p}\ar[r]^(.38){\delta }
  &{\Sigma^{2p^{n}-1}k (n)}\ar[r]^(.5){j}
     &{\Sigma^{2p^{n}-1}H/p}
}
\end{split}
\end{numequation}%

\noindent   is the Milnor primitive operation operation $Q_{n}$.

With the maps in  (\ref{eq-fiber-seq}) we can
construct the following {\em Adams diagram for $k (n)$},
\begin{numequation}\label{eq-kn-Adams-diagram}
\begin{split}
\xymatrix
@R=8mm
@C=12mm
{
{k (n)}\ar[d]_(.5){j}
  &{\Sigma^{|v_{n}|}k(n)}\ar[d]_(.5){j}
        \ar[l]_(.5){v_{n}}
    &{\Sigma^{2|v_{n}|}k(n)}\ar[d]_(.5){j}
        \ar[l]_(.5){v_{n}}
      &{\Sigma^{3|v_{n}|}k(n)}\ar[d]_(.5){j}
        \ar[l]_(.5){v_{n}}
        &{\dotsb }\ar[l]^(.5){}\\
{H/p}
  &{\Sigma^{|v_{n}|}H/p}
    &{\Sigma^{2|v_{n}|}H/p}
      &{\Sigma^{3|v_{n}|}H/p.}
}
\end{split}
\end{numequation}%

\noindent   Each fiber sequence
\begin{displaymath}
\Sigma^{(s+1) |v_{n}|}k (n)
  \to  \Sigma^{s|v_{n}|}k (n)
  \to  \Sigma^{s|v_{n}|}H/p
\end{displaymath}

\noindent leads to a long exact sequence of homotopy groups.  The same
is true if \crr{we} apply either the functor $F (X,-)$, the {\em cohomological
case}, or $(X\wedge -)$, the {\em homological case}, to
(\ref{eq-kn-Adams-diagram}).

In each case these long exact sequences assemble into an exact couple
(see \cite[\S2.1]{Rav:MU}) leading to a {\SS} $\left\{E^{s,t}_{r}
\right\}$, where
\begin{itemize}
\item [$\bullet$] $E_{1}^{s,t}$ is either $\pi_{t-s} (F ( X,
\Sigma^{s|v_{n}|}H/p))=H^{s (2p^{n}-1)-t}X$, the indicated mod $p$
cohomology group of $X$, or $\pi_{t-s} (X\wedge
\Sigma^{s|v_{n}|}H/p)=H_{t-s (2p^{n}-1)}X$, the indicated mod $p$
homology group of $X$.

\item [$\bullet$] $E_{2}^{s,t}$ is either $\Ext_{E
(Q_{n})}^{s,t}\left(\Z/p, H^{*}X\right)$ or $\Ext_{E
(Q_{n})}^{s,t}\left(H^{*}X, \Z/p\right)$.  This can be derived from
(\ref{eq-Qn}).

\item [$\bullet$] $d_{r}:E_{r}^{s,t}\to E_{r}^{s+r,t+r-1}$.  The
filtration index $s$ is raised by $r$ and the dimension index $t-s$ is
lowered by one.

\item [$\bullet$] $E_{\infty }^{s,t}$ is a certain subquotient of
either  $k (n)^{s-t}X$ or $k (n)_{t-s}X$.
\end{itemize}

\noindent These are the {\em classical {\ASS}s for $k (n)^{*}X$ and
$k (n)_{*}X$.}  It is common to depict them in a chart in which
$E_{r}^{s,t}$ has Cartesian coordinate $(t-s,s)$. Thus $d_{r}$ is an
arrow lowering the first coordinate by 1 and raising the second by
$r$, making it a line with slope $-r$.

The {\em {\ASS} for $\k^{*}(K_{2})$} has
\begin{displaymath}
E_{2}^{s,t} =
\Ext_{\mathcal{A}}^{s,t}(H^{*}(\k),H^{*}K_{2})\cong
\Ext_{\E}^{s,t}(\zp,H^{*}K_{2}) \implies \k^{-(t-s)}(K_{2}).
\end{displaymath}

\noindent We use the usual grading for the {\ASS} so that $E_r^{s,t}$
is displayed with Cartesian coordinates $(t-s,s)$, but then we give
the negative $x$-axis positive degrees, rewriting $E_r^{s,t}$ as
$G^r_{s-t,s}$ in position $(s-t,s)$.  We use $d^r$ for our cohomology
differentials.  {\em In this depiction, differentials raise rather
than lower the first coordinate by 1.}

We also have the {\em {\ASS} for $\k_{*}(K_{2})$}, and need to have
distinct notation to clearly separate it from the cohomology notation.
It has
\begin{displaymath}
E_{2}^{s,t} =
\Ext_{\mathcal{A}}^{s,t}(H^{*}(\k\wedge K_{2}),\zp)
\cong \Ext_{\E}^{s,t}(H^{*}K_{2},\zp)
\implies \k_{t-s}(K_{2}).
\end{displaymath}
We use the usual grading for the {\ASS} so that $E_r^{s,t}$
is at the Cartesian $(t-s,s)$.  Here we don't need the negative
grading, but to distinguish this from the cohomology {\ASS}, we write
$E_r^{s,t}$ as $G_r^{t-s,s}$ in position $(t-s,s)$.  Here we use $d_r$
for the differential so we can keep track of which is which.

To summarize, in the cohomological case
% Defined in ~/math/diary/diary.sty
\begin{numequation}\label{eq-cohom-case}
\begin{split}
G^{r}_{x,y}&:= E_{r}^{y,y-x}\qquad \mbox{with differentials }
  d^{r}: G^{r}_{x,y}\to G^{r}_{x+1,y+r},
\end{split}
\end{numequation}%

\noindent and in the homological case,
% Defined in ~/math/diary/diary.sty
\begin{numequation}\label{eq-hom-case}
\begin{split}
G_{r}^{x,y}
&:= E_{r}^{y,y+x}\qquad \mbox{with differentials }
  d_{r}: G_{r}^{x,y}\to G_{r}^{x-1,y+r}.
\end{split}
\end{numequation}%

We need the $\E$-module structure of $H^{*}X$ (which we will describe
in the next section) in order to compute the $E_{2}$-terms.  Any
$\E$-module $M$ is the sum of a free module and $\Z/p$-vector space on
which $Q_{n}$ acts trivially.  As a result, it is easy to compute the
relevant Ext groups. We have
% Defined in ~/math/diary/diary.sty
\begin{numequation}\label{eq-Ext-Qn}
\left\{\begin{split}
\Ext_{E (Q_{n})}^{*,*}\left(\Z/p, \Z/p\right)
 & = P (v_{n})  \qquad \mbox{with }v_{n}\in \Ext_{}^{1,2p^{n}-1},\\
\Ext_{E (Q_{n})}^{s,t}\left(E (Q_{n}), \Z/p\right)
 & = \mycases{
\Z/p   &\mbox{for } (s,t)= (0,0)\\
0      &\mbox{otherwise,}
}\\
\aand
\Ext_{E (Q_{n})}^{s,t}\left(\Z/p,  E (Q_{n})\right)
 & = \mycases{
\Z/p   &\mbox{for } (s,t)= (1,2p^{n}-1)\\
0      &\mbox{otherwise.}
}
\end{split} \right.
\end{numequation}%

\section{The $Q_n$ homology of $H^{*}K_{2}$ and the
Adams $E_{2}$ term} \label{qhomsec}

Following Tamanoi, \cite[Theorem 5.2]{Tam-knp}, we have, at odd
primes, $u_i = Q_i \iota_{2}$ and $z_i = Q_i u_{0}$ (in particular
$z_{0}=0$), giving us $H^{*}K_{2}$ as in (\ref{coho}), where $Q_{i}$
again is the $i$th Milnor primitive.  Continuing to follow Tamanoi, we
have
\begin{align*}
Q_{n}\iota_{2}
 & = u_{n}\\
Q_{n}u_{s}
 & =  \mycases{
z_{n-s}^{p^{s}}
       &\mbox{for }0\leq s<n\\
0      &\mbox{for }s=n\\
z_{s-n}^{p^{n}}
       &\mbox{for }s>n
} \\
Q_{n}z_{s}
 & = 0.
\end{align*}

To compute the $Q_{n}$ homology we filter $H^{*}K_{2}$ by powers of
the ideal
\begin{displaymath}
\left(y_{n,0},u_{s},z_{s}^{e_{n} (s)}:s\geq 0 \right),
\end{displaymath}

\noindent where $z_{0}=0$ and $e_{n} (s)$ is as in (\ref{eq-convenient}).
This means $z_{s}  \in F^{0}$ and $z_{s}^{e_{n} (s)} \in F^{1}$ for $s>0$, and
$u_{s}  \in F^{1}$  for $s\geq 0$.
The associated bigraded object $E_{0}H^{*}K_{2}$ and its
$Q_{n}$ homology are indicated in the following diagram, which uses the notation of (\ref{eq-convenient}).
% Defined in ~/math/diary/diary.sty
\begin{numequation}\label{eq-bigraded}
\begin{split}
\xymatrix
@R=0mm
@C=5mm
{
{L_{n}\otimes TZ_{n,0}}
 \ar@{~>}[r]^(.5){}
    &{L_{n}\otimes TZ_{n,0}}\\
{\otimes }
  &{\otimes }\\
{P([y_{n,0}])\otimes E([u_{n}])}\ar@{~>}[r]^(.5){}
    &{P([y_{n,1}])\otimes E\left([y_{n,1/2}] \right)}\\
{\otimes }
  &{\otimes }\\
{E([u_{s}]:0\leq s<n)\otimes
         P\left(\left[z_{s}^{p^{n-s}}\right]:0<s\leq n \right)}
                 \ar@{~>}[r]^(.5){}
      &{\Z/p}\\
{\otimes }
  &{\otimes }\\
{E([u_{2n+s}]:s>0)\otimes P\left(\left[z_{n+s}^{p^{n}}\right]:s>0 \right)}
            \ar@{~>}[r]^(.5){}
          &{\Z/p}\\
{\otimes }
  &{\otimes }\\
{E_{0}W_{n,0}}
            \ar@{~>}[r]^(.5){}
          &{E_{0}W_{n,0},}
}
\end{split}
\end{numequation}%
\noindent where elements enclosed in square brackets are in
$E_{0}H^{*}K_{2}$ (where they are indecomposable) corresponding to
unbracketed elements in $H^{*}K_{2}$, which may be decomposable.  The
elements $[y_{n,1}]$ and $[y_{n,1/2}]$ are in $F^{p}$,
and the other named elements below the top row are in $F^{1}$.

For the second row of (\ref{eq-bigraded}) we have an additive isomorphism
%\begin{displaymath}
\[
P(y_{n,0})\otimes E(u_{n})
\cong T_{p} (y_{n,0})\otimes E(w_{n,0}) \otimes P(y_{n,1}),
%\end{displaymath}
\]
\noindent and the behavior of the first two factors  is
illustrated in the upper diagram of (\ref{eq-diagrams}) below for $j=0$,
where $\rho_{n} (0)=1$.
This can be done now with explicit computation, eliminating the need for
the filtration on this part.

Using the notation of (\ref{eq-convenient}), we can
consolidate the third and fourth rows of (\ref{eq-bigraded}), and
rewrite them as
\begin{displaymath}
\xymatrix
@R=0mm
@C=10mm
{
{EE_{n}\otimes PZ_{n}}\ar@{~>}[r]^(.5){}
      &{\Z/p}.
}
\end{displaymath}
Because we end up with a trivial result, we can also eliminate the
need for the filtration here as well.

Thus we can rewrite (\ref{eq-bigraded}) as
% Defined in ~/math/diary/diary.sty
\begin{numequation}\label{eq-bigraded2}
\begin{split}
\xymatrix
@R=1mm
@C=5mm
{
{L_{n}\otimes T_{p^{n}} (z_{n+i}:i>0)}
 \ar@{~>}[r]^(.5){}
  &{L_{n}\otimes T_{p^{n}} (z_{n+i}:i>0)}\\
{\otimes }
  &{\otimes }\\
\crr{{T_{p} (y_{n,0})\otimes E(w_{n,0})}}
  \ar@{~>}[r]^(.5){}
    &{ E(y_{n,1/2})}\\
{\otimes }
  &{\otimes }\\
\crr{{ P(y_{n,1})}}
  \ar@{~>}[r]^(.5){}
    &{  P(y_{n,1})}\\
{\otimes }
  &{\otimes }\\
{EE_{n}\otimes PZ_{n}}
                 \ar@{~>}[r]^(.5){}
      &{\Z/p}\\
{\otimes }
  &{\otimes }\\
{W_{n,0}}
            \ar@{~>}[r]^(.5){}
          &{W_{n,0}.}
}
\end{split}
\end{numequation}%

\begin{thm}\label{E2k}
For $G^{2}_{*,*}$ as in (\ref{eq-cohom-case}), we have elements
\begin{align*}
v_{n}
 & \in G^2_{-2(p^{n}-1),1},&
y_{n,1}
 & \in G^2_{2p,0},&
w_{n,i}
 & \in G^2_{2p^{n+i}+1,0},\\
y_{n,1/2}
 &  \in G^2_{2(p^{n}-1)+2p+1,0},&
\mbox{and } z_{j}
 & \in G^2_{2(p^{j}+1),0}.
\end{align*}

\noindent
The $E_{2}$ term of the odd primary Adams
spectral sequence for $\k^{*}(K_{2})$ is
\begin{displaymath}
P(v_{n})\otimes L_{n}
\ot P(y_{n,1})
\ot
E(y_{n,1/2})
\otimes W_{n,0}
\otimes  TZ_{n,0}
\end{displaymath}

\noindent \crbr{plus $S_{n,0}\otimes M_{n,0}$ from \S\ref{subsec-d1},
the elements
annihilated by $v_{n}$. }
\end{thm}

\begin{proof}
The $Q_n$ homology of $H^{*}K_{2}$ gives us the trivial $\E$-module part.
The rest is free over $\E$. The  $\Ext$ groups for
both kinds of modules are as in (\ref{eq-Ext-Qn}).  The result follows.
\end{proof}

\section{Illustration for $n=2$}\label{sec-illustration}

%\todo[inline,color=green]{6/19/24. This section is new.}

In this section we will sometimes write our generators other than
$\iota_{2}$ with subscripts enclosed in parentheses indicating their
dimensions.  In the cohomological case we have $|v_{2}|=2-2p^{2}$,
which is $-16$ for $p=3$.

Thus we have
\begin{align*}
H^{*}K_{2}\,\,\,
 & \!\!\!=P(\iota_{2})
\otimes E(u_{s}:s\geq 0 )
\otimes P(z_{s}:s>0 ) \\
 & \!\!\!=P(\iota_{2})
\otimes E(u_{(3)},\, u_{(7)},\,u_{(19)},\,\dotsc  )
\otimes P(z_{(8)},\,z_{(20)},\,z_{(56)},\,\dotsc )
     \qquad \mbox{for }p=3
\end{align*}

\noindent with
% Defined in ~/math/diary/diary.sty
\begin{equation}\label{eq-Q2}
\begin{split}
Q_{2} (\iota_{2})
 & = u_{2},\\
Q_{2} (u_{s})
 & =\mycases{
z_{2}  &\mbox{for }s=0\\
z_{1}^{p}
       &\mbox{for }s=1\\
0      &\mbox{for }s=2\\
z_{s-2}^{p^{2}}
       &\mbox{for }s\geq 3,
}\\
\aand
Q_{2} (z_{s})
 &=0.
\end{split}
\end{equation}%

\noindent The actions of $Q_{2}$ on the first five $u_{s}$s
imply that
\begin{align*}
Q_{2} (u_{3}-z_{1}^{p^{2}-p}u_{1})
 & =0\\
\aand
Q_{2} (u_{4}-z_{2}^{p^{2}-1}u_{0})
 & =0,
\end{align*}

\noindent so as in (\ref{eq-w}) we define
\begin{displaymath}
w_{2,0}
  := u_{2},\qquad
w_{2,1}
 := u_{3}-z_{1}^{p^{2}-p}u_{1}\qquad
\aand
w_{2,2}
   := u_{4}-z_{2}^{p^{2}-1}u_{0},
\end{displaymath}

\noindent with each being killed by $Q_{2}$.

The Adams $E^{1}$-term is
\begin{align*}
\lefteqn{P(v_{2})\otimes P(\iota_{2})\otimes E(u_{s}:s\geq 0)
      \otimes P(z_{s}:s\geq 1)}\qquad\qquad\\
\mbox{with }\qquad \qquad
v_{2} & \in G^{2}_{2-2p^{2},1} &
\iota_{2}
  & \in G^{2}_{2,0}\\
u_{s}
  & \in G^{2}_{2p^{s}+1,0}&
z_{s}
  & \in G^{2}_{2p^{s}+2,0}\\
\aand
d^{1}\iota_{2}
 & = v_{2} u_{2}  &
d^{1}u_{0}
 & = v_{2} z_{2}  \\
d^{1}u_{1}
 & = v_{2} z_{1}^{p}  &
d^{1}u_{s}
 & = v_{2} z_{s-2}^{p^{2}}\qquad \mbox{for }s\geq 3.
\end{align*}

It follows that modulo $v_{2}$-torsion, the Adams $E_{2}$-term is
% Defined in ~/math/diary/diary.sty
\begin{numequation}\label{eq-E2-mod-torsion}
\begin{split}
\lefteqn{P(v_{2})\otimes P(\iota_{2}^{p})
      \otimes E(\iota_{2}^{p-1}u_{2}, w_{2,1},w_{2,2})
      \otimes T_{p}(z_{1})\otimes T_{p^{2}}(z_{s}: s\geq 3)}\qquad\qquad\\
 & = k (2)^{*}\otimes P(y_{2,1})
      \otimes E(y_{2,1/2}, w_{2,1},w_{2,2})
      \otimes L_{2}\otimes T_{p^{2}}(z_{s}: s\geq 3).
\end{split}
\end{numequation}%

\begin{lemma}\label{lem-E2}
For any prime and for all $n$,
in the {\ASS}  for $k (n)^{*}K_{2}$,
\begin{enumerate}[label={(\roman*)},itemindent=1em]
\item \label{lem-E2ii}
every power of $\iota_{2}$ supports a differential, and
\item  \label{lem-E2i}
$z_{s}$ is a nontrivial permanent cycle for $s>0$,
\item \label{lem-E2iii}
some $v_{n}$-multiple of each $z_{s}$ for $s> n$ is killed by a
differential.
\end{enumerate}
\end{lemma}

Lemma \ref{limita} below is a similar statement.

\proof
\begin{enumerate}[label={(\roman*)},itemindent=1em]
\item There is a fiber sequence
\begin{displaymath}
K (\Z,2)\to K_{2}\to K (\Z,3)
\end{displaymath}

\noindent for which the {\SSS } collapses, that is
\begin{align*}
H^{*}K (\Z,2)
 & = P (\iota_{2})  \\
H^{*}K (\Z,3)
 & = E (u_{s}:s\geq 0)
      \otimes P (z_{s}:s\geq 1) \\
H^{*}K_{2}
 & = H^{*}K (\Z,2)\otimes H^{*}K (\Z,3) .
\end{align*}

\noindent This means that nothing in $P (\iota_{2})$ can be hit by an
Adams differential for any $n$.  Thus (\ref{eq-K2K2}) below implies that
each power of $\iota_{2}$ must support a nontrivial Adams differential.
\crbr{For $n=2$,}
the action of $Q_{2}$ in $H^{*}K (\Z,3)$ is given in (\ref{eq-Q2}).

\bigskip
\item We also have a $p$-local fiber sequence
\begin{displaymath}
K (\Z,3)\to BP\langle 1 \rangle_{2p+2}\to BP\langle 1 \rangle_{4}
\end{displaymath}

\noindent in which the second and third spaces have even dimensional
cohomology.  The generators $z_{i}\in H^{*}K (\Z,3)$ are in the image
of the map from\linebreak $H^{*}BP\langle 1 \rangle_{2p+2}$, so they
map to permanent cycles in the {\ASS}s for both $k (n)^{*}K (\Z,3)$
and $k (n)^{*}K_{2}$.

\bigskip
\item We know by \cite[dual to Theorem 11.1]{RW:CF} that
% Defined in ~/math/diary/diary.sty
\begin{equation}\label{eq-K2K2}
\begin{split}
K (n)^{*}K_{2} = K (n)^{*}\otimes \bigotimes_{0<s<n} T_{p^{n-s}}(z_{s}).
\end{split}
\end{equation}%
\qedhere

\noindent This means that our Adams $E_{\infty }$-term must be
congruent  to
\begin{displaymath}
k (n)^{*}\otimes \bigotimes_{0<s<n} T_{p^{n-s}}(z_{s})
\end{displaymath}

\noindent modulo $v_{n}$-torsion.

\end{enumerate}

\bigskip {\em It turns out that there is only one pattern of higher
Adams differentials that leads to an answer meeting the conditions
imposed by Lemma \ref{lem-E2}.}

\bigskip
For $p=3$, it begins as follows.
\begin{equation}\label{eq-diffs}
\begin{split}
\begin{array}{rllrll}
[d^{1} (\iota_{2})
 & =v_{2} w_{(19)}& \in G^{1}_{3,1}]  &
d^{2} (\iota_{2}^{2}w_{(19)})
 & =v_{2}^{2}z_{(56)}& \in G^{2}_{24,2}   \\
d^{3} (\iota_{2}^{3})
 & =v_{2}^{3} w_{(55)}& \in G^{3}_{7,3}  &
d^{6} (\iota_{2}^{6}w_{(55)})
 & =v_{2}^{6}z_{(164)}& \in G^{6}_{68,6}   \\
d^{9} (\iota_{2}^{9})
 & =v_{2}^{9} w_{(163)}& \in G^{9}_{19,9}  &
\qquad d^{18} (\iota_{2}^{18}w_{(163)})
 & =v_{2}^{18}z_{(488)}& \in G^{18}_{200,18}
\end{array}
\end{split}
\end{equation}%

\begin{remark}\label{rem-2}
{\bf Standard notational abuse.}  In (\ref{eq-diffs}) we are abusing
notation for higher differentials in the usual way.  For example, the
source of the $d^{2}$, is written as $\iota_{2}^{2}u_{(19)}$. However it
is not really a product in $G^{2}$ because $\iota_{2}$ is no longer
present there since it supported a $d^{1}$. Strictly speaking,
$\iota_{2}^{2}u_{(19)}$ is an abbreviation for the Massey product
\begin{displaymath}
\langle v_{2}^{2} ,\,u_{(19)} ,\,u_{(19)}  ,\,u_{(19)} \rangle
\in G^{2}_{23,0}.
\end{displaymath}%

\noindent Similarly $\iota_{2}^{3}$ is code for
\begin{displaymath}
\langle u_{(19)},\,v_{2} ,\,v_{2} u_{(19)},\,v_{2},\,u_{(19)} \rangle
\in G^{2}_{6,0}=G^{3}_{6,0}.
\end{displaymath}

An introduction to Massey products can be found in \cite[A1.4]{Rav:MU}
(and in \cite[A1.4]{Rav:MU2}), which is an introduction to Peter May's
definitive paper on the subject \cite{May:Matric}.
\end{remark}

Let
\begin{displaymath}
y_{2,s}:=\iota_{2}^{p^{s}}\qquad \mbox{for }s\geq  0
\qquad \aand
y_{2,s+1/2}:=y_{2,s}^{p-1}w_{2,s}\qquad \mbox{for }s\geq 0 .
\end{displaymath}

% \noindent
% % Defined in ~/math/diary/diary.sty
% \begin{equation}\label{eq-y}
% \begin{split}
% y_{2,i}:=\iota_{2}^{p^{i}}\qquad \mbox{for }i\geq  0
% \qquad \aand
% y_{2,i+1/2}:=y_{2,i}^{p-1}w_{2,i}\qquad \mbox{for }i\geq 0 .
% \end{split}
% \end{equation}%

\noindent Then for a general prime $p$, (\ref{eq-diffs}) reads
\begin{equation}\label{eq-diffsp}
\begin{split}
\begin{array}{rlrl}
d^{1} (y_{2,0})
 & =v_{2} w_{2,0}
   &d^{p-1} (y_{2,1/2})
     & =v_{2}^{p-1}z_{3}\\
d^{p} (y_{2,1})
 & =v_{2}^{p} w_{2,1}
   &d^{p^{2}-p} (y_{2,3/2})
     & =v_{2}^{p^{2}-p}z_{4}\\
d^{p^{2}} (y_{2,2})
 & =v_{2}^{p^{2}} w_{2,2}\qquad
   &d^{p^{3}-p^{2}} (y_{2,5/2})
     & =v_{2}^{p^{3}-p^{2}}z_{5}
\end{array}
\end{split}
\end{equation}%

The first differential reflects the fact that
\begin{displaymath}
Q_{2}y_{2,0} = Q_{2}\iota_{2} = u_{2}=w_{2,0} \in H^{*}K_{2}.
\end{displaymath}

\noindent We also have
\begin{displaymath}
Q_{2}u_{s}=\mycases{
z_{2-s}^{p^{s}}
       &\mbox{for }0\leq s\leq 1\\
0      &\mbox{for }s=2\\
z_{s-2}^{p^{2}}
       &\mbox{for }s\geq 3.
}
\end{displaymath}

\noindent These lead to
\begin{displaymath}
E_{2}\cong T_{p} (z_{1})\otimes P (v_{2}, y_{2,1})
     \otimes  E (y_{2,1/2},w_{2,1},w_{2,2})\otimes T_{p^{2}} (z_{2+s}:s>0)
\end{displaymath}

\noindent modulo $v_{2}$-torsion as in (\ref{eq-E2-mod-torsion}), where
\begin{displaymath}
y_{2,1/2}=y_{2,0}^{p-1}w_{2,0}.
\end{displaymath}

\noindent  For $p=3$, this reads
\begin{displaymath}
E_{2}\cong T_{3} (z_{(8)})\otimes P (v_{2}, y_{(6)})
     \otimes  E (y_{(23)}, w_{(55)},w_{(163)})
       \otimes T_{9} (z_{(56)}, z_{(164)}, z_{(488)},\dotsc ).
\end{displaymath}

\noindent Lemma \ref{lem-E2}\ref{lem-E2iii} requires a differential
hitting a $v_{2}$-multiple of $z_{(56)}$.  It cannot be supported by a
$v_{2}$-torsion element since its target is torsion free.  The only
classes in low enough dimensions live in $ T_{p} (z_{1})\otimes P
(y_{2,1}) \otimes E (y_{2,1/2})$.  Since $z_{1}$ is a permanent
cycle, we can restrict our attention to $P (y_{2,1}) \otimes E
(y_{2,1/2})$, which is $P (y_{(6)})\otimes E (y_{(23)})$ for
$p=3$.  The only class in a dimension congruent to $|z_{3}|-1$ mod
$|v_{2}|$ (55 mod 16 for $p=3$) is $y_{2,1/2}$, which is $y_{(23)}$
for $p=3$.

{\em This gives us the second differential listed in (\ref{eq-diffsp}).}
It also gives
\begin{displaymath}
E_{p}\cong T_{p} (z_{1})\otimes P (v_{2},y_{2,1})
     \otimes  E (w_{2,1},w_{2,2}, w_{2,3})\otimes T_{p^{2}} (z_{2+s}:s>1)
\end{displaymath}

\noindent modulo $v_{2}$-torsion, which for $p=3$ reads
\begin{displaymath}
E_{3}\cong T_{3} (z_{(8)})\otimes P (v_{2}, y_{(6)})
     \otimes  E ( w_{(55)},w_{(163)}, w_{(471)})
       \otimes T_{9} (z_{(164)}, z_{(488)},z_{(1460)},\dotsc ).
\end{displaymath}

What happens next? Something has to kill $z_{(164)}$, but none of the
listed lower dimensional generators are in the right dimension to do
so.  However if\linebreak  $d^{3}y_{(6)}=v_{2}^{3}w_{(55)}$, we would get a new
generator $y_{2,3/2}=y_{2,1}^{p-1}w_{2,1}$, which is\linebreak
$y_{(67)}=y_{(6)}^{2}w_{(55)}$ at $p=3$.  It is in the right dimension
to kill $v_{2}^{6}z_{(164)}$.  Thus we get the next two differentials
listed in (\ref{eq-diffsp}) and we have (modulo $v_{2}$-torsion)
\begin{align*}
E_{p+1}=E_{p^{2}-p}
 & \cong  T_{p} (z_{1})\otimes P (v_{2},y_{2,2})
     \otimes  E (y_{2,3/2},w_{2,2}, w_{2,3})
    \otimes T_{p^{2}} (z_{2+i}:i>1)  \\
\aand
E_{p^{2}-p+1}
 & \cong  T_{p} (z_{1})\otimes P (v_{2},y_{2,2})
     \otimes  E (w_{2,2}, w_{2,3}, w_{2,4})
    \otimes T_{p^{2}} (z_{2+i}:i>2)\\
 &\qquad \mbox{with }w_{2,4}=y_{2,3/2}z_{4}^{p^{2}-1}  .
\end{align*}

\noindent For $p=3$, this reads
\begin{align*}
E_{3}=E_{6}
 & \cong  T_{3} (z_{(8)})\otimes P (v_{2},y_{(18)})
     \otimes  E (y_{(67)},w_{(163)}, w_{(471)})
    \otimes T_{9} (z_{(164)}, z_{(488)}, \dotsc )  \\
\aand
E_{7}
 & \cong  T_{3} (z_{(8)})\otimes P (v_{2},y_{(18)})
     \otimes  E (w_{(163)}, w_{(471)}, w_{(1379)})
    \otimes T_{9} (z_{(488)}, z_{(1460)}, \dotsc )   .
\end{align*}

Note that at each stage we have the following factors:

\begin{itemize}
\item [$\bullet$] $k (2)^{*}\otimes T_{p} (z_{1})$,
\item [$\bullet$] the polynomial algebra generated by some $p^{k}$th
power of $\iota_{2}$,
\item [$\bullet$] an exterior algebra on three generators (two $w$s
plus a $y$ or a third $w$) with each having a dimension congruent to 3
or $2p+1$ mod $|v_{2}|$, and
\item [$\bullet$] a truncated polynomial algebra of height $p^{2}$ on
infinitely many $z_{i}$s having dimensions alternately congruent to 4
and $2p+2$ mod $|v_{2}|$.
\end{itemize}

The exterior generator with the $y$ label is the only one in a
position to kill the next $z$.  {\em These phenomena persist
throughout the {\SS} and generalize to larger values of $n$.}  The
factor $T_{p} (z_{1})$ generalizes to $L_{n}$ as in
(\ref{eq-convenient}).  The dimension of each $y_{n,i}$ is congruent
to 2 modulo $2p-2$, while those of the exterior generators and the
$z_{i}$s are congruent to 3 and 4 respectively.  Half the generators
remove $y_{n,i}$ and $w_{n,i}$, replacing them with $y_{n,i+1}$ and
$y_{n,i+1/2}$. The others remove $z_{n+i+1}$ and replace $y_{n,i+1/2}$
by $w_{n,i+n+1}$.

We want to extend (\ref{eq-diffsp}) further with differentials
supported by higher powers of $\iota_{2}$ in the left column and ones
killing $v_{2}$-multiples of higher $z_{i}$s in the right column.

Since $v_{2}z_{3}^{p^{2}}=0$ in $G^{2}$, the element
\begin{displaymath}
w_{2,3} := y_{2,0}^{p-1}w_{2,0}z_{3}^{p^{2}-1}
=\langle v_{2}^{p-2},\,v_{2}z_{3} ,\,z_{3}^{p^{2}-1} \rangle
\qquad \mbox{as in (\ref{eq-w})}
\end{displaymath}

\noindent is a $d^{p-1}$-cycle and hence a target for $y_{2,3}$.
Thus we have
\begin{align*}
d^{p^{3}-p+2}y_{2,3}
 & = v_{2}^{p^{3}-p+2}w_{2,3} &
\aand
d^{p^{4}-p^{3}+\crr{p-2}} (y_{2,3}^{p-1}w_{2,3})
 & = v_{2}^{p^{4}-p^{3}+p-2}z_{6}
\end{align*}

\noindent with
\begin{displaymath}
y_{2,7/2}:=y_{2,3}^{p-1}w_{2,3}
 =\langle v_{2}^{(p-1) (p^{3}-1)},\,
     \underbrace{w_{2,3},\,\dotsc ,\,w_{2,3}}_{\text{$p$ factors} } \rangle.
\end{displaymath}

We denote the indices of the
differentials on $y_{2,i}$ and $y_{2,i+1/2}$ by $\rho_{2} (i)$ and\linebreak
$\rho_{2} (i+1/2)$.  Hence the $i$th row of (\ref{eq-diffsp}) is
\begin{displaymath}
d^{\rho_{2} (i)}y_{2,i}
 = v_{2}^{\rho_{2} (i)}w_{2,i}
   \qquad \aand  d^{\rho_{2}(i+1/2)}y_{2,i+1/2}
      = v_{2}^{\rho_{2}(i+1/2)}z_{i+n+1},
\end{displaymath}

% \noindent
% % Defined in ~/math/diary/diary.sty
% \begin{equation}\label{eq-ith-row}
% \begin{split}
% \begin{array}{rlrl}
% d^{\rho_{2} (i)}y_{2,i}
%  & = v_{2}^{\rho_{2} (i)}w_{2,i}
%    &\aand  d^{\rho_{2}(i+1/2)}y_{2,i+1/2}
%       & = v_{2}^{\rho_{2}(i+1/2)}z_{i+3},
% \end{array}
% \end{split}
% \end{equation}%

\noindent where
\begin{displaymath}
w_{2,i}=\mycases{
u_{2}  &\mbox{for }i=0\\
u_{3}-z_{1}^{p^{2}-p}u_{1}
       &\mbox{for }i=1\\
u_{4}-z_{2}^{p^{2}-1}u_{0}
       &\mbox{for }i=2\\
y_{2,i-3}^{p-1}w_{i-1}z_{i}^{p^{2}-1}
 =y_{2,i-5/2}z_{i}^{p^{2}-1}
       &\mbox{for }i\geq 3
}
\end{displaymath}

\noindent as in (\ref{eq-w}).

The following is a special case of Lemma \ref{lem-ab} below.

\begin{prop}\label{prop-n2rho}
The indices $\rho_{2} (i)$ and $\rho_{2} (i+1/2)$ for integers $i\geq 0$ are
\begin{align*}
\rho_{2} (i)
 & =  \mycases{
p^{i}
       &\mbox{for }0\leq i\leq 2\\
 p^{i}-p^{i-2}+1+\rho_{2} (i-3)
       &\mbox{for }i\geq 3}  \\
\aand \rho_{2} (i+1/2)
 & = p^{i+1}-\rho_{2} (i).
\end{align*}
\end{prop}

\section{Numbers and definitions}
\label{numdef}

% Although we are only interested in odd primes until the last  section,
% all of the formulas and numbers in the section work for $p=2$ except for
% one small deviation.

In this section we give some definitions and compute some numbers we
need.  We already have elements $y_{n,i}$, $z_{i}$ and $w_{n,i}$ and
$y_{n,i+1/2}$ with
% Defined in ~/math/diary/diary.sty
\begin{equation}\label{induct}
\begin{split}
|y_{n,i}| & = 2p^{i},\\
|z_{i}| & = 2 (p^{i}+1),\\
|y_{n,i+1/2}|
   &= |y_{n,i}^{p-1}w_{n,i}|= 2p^{i}(p-1)+| w_{n,i}|, \\
\aand
|w_{n,i+n+1}|
 & = |y_{n,i+1/2}z_{n+i+1}^{p^{n}-1}|
   = |y_{n,i}^{p-1}w_{n,i}z_{n+i+1}^{p^{n}-1}|\\
 & = 2p^{i}(p-1) + |w_{n,i}|+ 2(p^{n}-1) (p^{n+i+1}+1)\\
 & = 2p^{i} (p^{2n+1}-p^{n+1}+p-1)+2 (p^{n}-1) + |w_{n,i}|\\
 & = 2p^{i} (p^{n+1}c_{n}+c_{1})+2c_{n} + |w_{n,i}|,\\
 &\qquad \mbox{where }c_{k}:=p^{k}-1.
\end{split}
\end{equation}%

Regarding these as functions of $i$, we will see that each one
satisfies a recursive formula similar to that for $|w_{n,i}|$.  To
study such functions, we need some notation.

%\clearpage
\begin{defin}\label{def-floor}\phantom{Hello}
\begin{enumerate}[label={(\roman*)},itemindent=1em]
% \item \label{def-floori}For a rational number $x$, $\lfloor x
% \rfloor$, the {\em floor of $x$}, denotes the greatest integer not
% exceeding $x$.
% \item \label{def-floorii} For a fixed
% positive integer $n$, any integer $i\geq 0$ can be written uniquely as
% \begin{displaymath}
% i=i_{0}+ (n+1)i_{1}\qquad \mbox{with }0\leq i_{0}\leq n,
% \end{displaymath}

% \noindent making $i_{1}=\lfloor i/ (n+1) \rfloor$.

% \item \label{def-flooriii}Similarly any integer $i\geq 0$ can be written uniquely as
% \begin{displaymath}
% i=i'_{0}+ n i'_{1}\qquad \mbox{with }0\leq i'_{0}< n,
% \end{displaymath}

% \noindent making $i'_{1}=\lfloor i/n \rfloor$.

\item \label{def-floorii} For a fixed positive integer $n$, let
$i_{1}=\lfloor i/ (n+1) \rfloor$ and $i_{0}=i- (n+1)i_{1}$ \crb{(the
reduction of $i$ modulo $n+1$)} for any integer $i$.

\item \label{def-flooriii}Similarly let $i'_{1}=\lfloor i/n  \rfloor$ and $i'_{0}=i- ni'_{1}$.

\item \label{def-flooriv}   Let
\begin{align*}
g_{n} (i)
 &: = \frac{p^{i}-p^{i_{0}}}{p^{n+1}-1}
 = \mycases{
0    &\mbox{for }0\leq i\leq n\\
p^{i-n-1}+g_{n} (i-n-1)
     &\mbox{for }i\geq n+1
   }\\
 &\phantom{:} = p^{i- (n+1)}+p^{i-2 (n+1)}+p^{i-3 (n+1)}+\dotsb +p^{i_{0}}.
\end{align*}
\end{enumerate}

It follows that $i$ can be written uniquely as
\begin{displaymath}
i=\mycases{
i_{0}+ (n+1)i_{1}&\mbox{with }0\leq i_{0}\leq n\\
%\aand
i'_{0}+ ni'_{1}&\mbox{with }0\leq i'_{0}< n
}
\end{displaymath}

% \noindent
% \begin{equation}\label{eq-q}
% \begin{split}
% g_{n} (i)
% % &:=p^{j}q_{n} (k)= \frac{p^{i}-p^{j}}{p^{n+1}-1}\\
%  &: = \frac{p^{i}-p^{i_{0}}}{p^{n+1}-1} \\
%  & = \mycases{
% 0    &\mbox{for }0\leq i\leq n\\
% p^{i-n-1}+g_{n} (i-n-1)
%      &\mbox{for }i\geq n+1
%    }
% \end{split}
% \end{equation}%
\end{defin}

%The proof of the following is elementary and  left to the reader.

\begin{lemma}\label{lem-f}
For a fixed positive integer $n$, suppose we have an integer valued
function $f_{n} (i)$ defined for integers $i\geq 0$ and satisfying the
recursive equation
\begin{displaymath}
f_{n} (i+n+1) = ap^{i}+f_{n} (i)+b \qquad
\mbox{for  constants $a$ and $b$.}
\end{displaymath}

\noindent Then, with notation as in Definition \ref{def-floor},
\begin{align*}
f_{n} (i)
 & = ag_{n} (i)+ b\lfloor i/ (n+1) \rfloor +f_{n} (i_{0})
\crr{\,\,=\, ag_n(i) + f_n(i_0)+bi_1}
\\
 & \equiv a\left(\frac{p^{i'_{0}}-p^{i_{0}}}{p-1} \right)
  +f_{n} (i_{0})+ b\lfloor i/ (n+1) \rfloor\qquad \bmod (p^{n}-1),
\end{align*}

\noindent and the latter expression is an integer.
\end{lemma}

\proof
Iterating the recursion relation gives
\begin{align*}
f_{n} (i)
 & = ap^{i- (n+1)}+f_{n} (i- (n+1))+b  \\
 & = a\left(p^{i- (n+1)}+p^{i- 2(n+1)} \right) +f_{n} (i- 2(n+1))+2b  \\
 & \hspace{2mm}\vdots    \\
 & = a \left(p^{i- (n+1)}+p^{i- 2(n+1)}+\dotsb +p^{i_{0}} \right)
             + f_{n} (i_{0})+ i_{1}b  \\
 & = ag_{n} (i) +f_{n} (i_{0})+\lfloor i/ (n+1) \rfloor.
\end{align*}

\noindent The congruence modulo $(p^{n}-1)$ follows from the fact that
$p^{n}\equiv 1$.  \qed

\begin{lemma}\label{lem-ab}
The values of $a$, $b$, $f_{n} (i_{0})$ and $f_{n} (i)\bmod 2
(p^{n}-1)$ for some functions of interest are shown in the following
table, where again $c_{k}:=p^{k}-1$.
%$\alpha :=p-1$ and $\beta :=p^{n}-1$.
\begin{center}
\begin{tabular}[]{|c||c|c|c|c|}
\hline
$f_{n} (i)$
  &$a$
    &$b$
      &$f_{n} (i_{0})$
        &$f_{n} (i)\bmod 2 c_{n}$\\
%        &$f_{n} (i)\bmod 2 (p^{n}-1)$\\
\hline
$g_{n} (i)$
  &$1$
    &$0$
      &$0$
        &\\
$\lfloor i/(n+1) \rfloor$
  &$0$
    &$1$
      &$0$
        &\\
$|y_{n,i}|=2p^{i}$
  &$2 c_{n+1}$
  %$2 (p^{n+1}-1)$
    &$0$
      &$2p^{i_{0}}$
        &$2p^{i'_{0}}$\\
$|u_{i}|=2p^{i}+1$
  &$2 c_{n+1}$
  %$2 (p^{n+1}-1)$
    &$0$
      &$2p^{i_{0}}+1$
        &$2p^{i'_{0}}+1$\\
$|z_{i}|=2p^{i}+2$
  &$2 c_{n+1}$
  %$2 (p^{n+1}-1)$
    &$0$
      &$2p^{i_{0}}+2$
        &$2p^{i'_{0}}+2$\\
\hline
$\rho_{n} (i)$
  &$pc_{n} $
    &$1$
      &$p^{i_{0}}$
        &\\
$\rho_{n} (i+1/2)$
  &$p^{n+1}c_{1} $
    &$-1$
      &$p^{i_{0}}c_{1} $
        &      \\
$|w_{n,i}|$
  &$2 (c_{1} +p^{n+1}c_{n} )$
    &$2 c_{n} $
      &$2p^{n+i_{0}}+1$
        &$2p^{i'_{0}}+1$\\
$|y_{n,i+1/2}|$
  &$2p^{n+1} (c_{1}+c_{n}) $
    &$2c_{n} $
      &$2p^{i_{0}} (c_{n}+p) +1$
        &$2p^{i'_{0}+1}+1$\\
\hline
\end{tabular}
\end{center}

In particular there are relations
% Defined in ~/math/diary/diary.sty
\crr{
\begin{numequation}\label{eq-wu}
\begin{split}
|w_{n,i}|\;\crr{\leq} \;|u_{n+i}|
 & =2p^{n+i}+1\\
\rho_n(i+1/2) +  \rho_n(i)
 &  = p^{i+1}, \\
 \aaand  \rho_n(i)
 &  \le p^i < \rho_n(i+1).
\end{split}
\end{numequation}%
}

\crb{More explicitly for integers $i\geq 0$,
% Defined in ~/math/diary/diary.sty
\begin{numequation}\label{eq-rho-n}
\begin{split}
\rho_{n} (i)
& = (p^{n+1}-p) (p^{i- (n+1)}+p^{i-2 (n+1)}+\dotsb +p^{i_{0}})\\
&\qquad
        + p^{i_{0}}+i_{1}\\
& =  p^{i}-p^{i-n}+p^{i-n-1}-p^{i-2n-1}+\dotsb\\
& \qquad  +p^{i_{0}+n+1}-p^{1+i_{0}}   + p^{i_{0}}     +i_{1}\\
& = \mycases{
p^{i}
       &\mbox{for }0\leq i\leq n\\
p^{i}-p^{i-n}+p^{i-n-1}+1
       &\mbox{for }n+1\leq i\leq 2n+1\\
p^{i}-p^{i-n}+p^{i-n-1}-p^{i-2n-1}+p^{i-2n-2}+2\hspace{-3cm}\\
       &\mbox{for }2n+2\leq i\leq 3n+2\\
\vdots
}\\
%\aand \rho_{n} (i+1/2)
%& = p^{i+1}-\rho_{n} (i).
\end{split}
\end{numequation}%
}
\end{lemma}

We do not need the values of $f_{n} (i)\bmod 2 (p^{n}-1)$ in the cases
where it is not shown.

% \begin{cor}\label{cor-}
% %{\bf .}
% $\rho_{n} (i)+\rho_{n} (i+1/2)=p^{i+1}$.
% \end{cor}

\proof[Proof of Lemma \ref{lem-ab}] The first five functions are
defined explicitly, so filling in the columns for them is
straightforward.  We also know the values $f_{n} (i_{0})$ for the last
four functions listed, so it remains to determine the constants $a$
and $b$ for each of them.  The congruences modulo $2 (p^{n}-1)$ are
also straightforward.

The constants $a$ and $b$ for $\rho_{2} (i)$ were given in Proposition
\ref{prop-n2rho}.

Our differentials
\begin{displaymath}
d^{\rho_{n}(i)} (  y_{n,i})  = v_{n}^{\rho_{n}(i)}w_{n,i}
\qquad\aand
d^{\rho_{n}(i+1/2)} ( y_{n,i+1/2}) = v_{n}^{\rho_{n}(i+1/2)}z_{n+i+1}.
\end{displaymath}

% \noindent
% \begin{equation}
% \label{twodiff}
% d^{\rho_{n}(i)} (  y_{n,i})  = v_{n}^{\rho_{n}(i)}w_{n,i}
% \qquad\aand
% d^{\rho_{n}(i+1/2)} ( y_{n,i+1/2}) = v_{n}^{\rho_{n}(i+1/2)}z_{n+i+1}.
% \end{equation}

\noindent imply
\begin{align}
\begin{split}
\label{diff}
 | y_{n,i} | + 1 + 2 (p^{n}-1)\rho_{n}(i)
 & = |w_{n,i}|\\
\aand | y_{n,i+1/2}| + 1 +  2 (p^{n}-1)\rho_{n}(i+1/2)
 & = |z_{i+n+1}|.
\end{split}
\end{align}

\noindent This means that the constants for $\rho_{n} (i)$ and
$\rho_{n} (i+1/2)$ are determined by those of $|y_{n,i}|$ and
$|z_{i}|$, which are known, and those of $|w_{n,i}|$ and
$|y_{n,i+1/2}|$, to which we now turn.

The constants $a$ and $b$ for $|w_{n,i}|$ are given by (\ref{induct}).

For $|y_{n,i+1/2}|$, (\ref{zhalf}) implies
\begin{align*}
y_{n,i+1/2+n+1}
 & = y_{n,i+n+1}^{p-1} w_{n,i+n+1}
   = y_{n,i+n+1}^{p-1}w_{n,i}y_{n,i}^{p-1}z_{i+n+1}^{p^{n}-1}\\
 & = y_{n,i+n+1}^{p-1}y_{n,i+1/2}z_{i+n+1}^{p^{n}-1}  ,   \\
\mbox{so}\qquad
|y_{n,i+1/2+n+1} |
 & =  |y_{n,i+1/2}|+|y_{n,i+n+1}^{p-1}z_{i+n+1}^{p^{n}-1}| \\
 & =  |y_{n,i+1/2}|+2 (p-1)p^{i+n+1}+2 (p^{n}-1) (p^{i+n+1}+1),
\end{align*}

\noindent which gives the stated values of $a$ and $b$.
\qed\bigskip

\section{Preliminaries before the proof}
\label{prelim}

Before proving Theorems \ref{uct} and \ref{pairing}, we have the
following observation.

\begin{prop}{\bf Divisibility Criterion.} \label{limit} If in the
 {\ASS} for $k (n)^{*}X$, $d^{r}(\alpha ) =
v_{n}^r \beta$, then
\begin{displaymath}
| \alpha | + 1 + 2r(p^{n}-1) = | \beta |,
\end{displaymath}

\noindent i.e., $|\beta |$ is congruent to $1+|\alpha |$ modulo
$|v_{n}|$.
\end{prop}

\begin{proof}[Proof of Theorem \ref{uct}] In \cite[Corollary 11.8]
{Laz-kn} and \cite[Theorem 2.3]{Rob-kn}, the odd primary $k(n)$ is
shown to be $A_\infty$.  In private communication, Lazarev says that
his argument for $k(n)$ works just as well for $p=2$.

In \cite[p. 257]{Rob-UCT}, Robinson produces a Universal Coefficient
Theorem for $A_\infty$ spectra.  In our case this gives the spectral
sequence of Theorem \ref{uct}.  For spaces of finite type with
$K(n)_{*}(X)$ finitely generated, $k(n)_{*}(X)$ is the sum of a free
module (of finite dimension) over $k(n)_{*}$ and a sum of torsion
modules, $T_k(v_{n})$.  The above $\Ext$ is easy to compute and
everything is in $\Ext^0$ and $\Ext^1$.  More precisely:
$$
\Ext_{k(n)_{*}}^{0,*}(k(n)_{*},k(n)_{*}) = k(n)^{*}
$$
$$
\Ext_{k(n)_{*}}^{1,*}(T_k(v_{n}),k(n)_{*}) = T_k(v_{n})
$$
$$
\text{with generator in} \  \Ext^{1,|v_{n}^k|}_{k(n)_{*}}
\  \text{and}\  v_{n}  \in \Ext_{k(n)_{*}}^{0,-2(p^{n}-1)}
$$
The entire $E_{2}$ term is in $\Ext^0$ and $\Ext^1$.  This is peculiar
to $k(n)$.  As a result, the spectral sequence collapses.
\end{proof}

\begin{proof}[Proof of Theorem \ref{pairing}] If we have
$d^\crr{r}(\alpha ) = v_{n}^r \beta $ in the {\ASS} for $k(n)^{*}(X)$,
it means we have (a cohomology) $T_r(v_{n})$ with generator in the
degree of $\beta $.  From the UCT, to get this, we must have a
(homology) $T_r(v_{n})$ with generator in the degree of $\alpha $.  To
get this in the {\ASS} for $k(n)_{*}(X)$, we must have a differential
$d_{\rho_{n}}(\beta ') = v_{n}^r \alpha '$ with the mentioned degrees.
Reverse the argument to get the other direction.
\end{proof}

\begin{remark}\label{rem-invert}
%{\bf .}
There is a way to invert $v_{n}$ in the {\ASS} which converts it to a full
plane {\SS }, the {\em localized} {\ASS}, rather than an upper half
plane one. Details can be found in \cite[\S2.3]{MRS4}.
\end{remark}

\begin{remark}\label{rem-synthetic}
It seems likely that Theorem \ref{pairing} also follows from the method of
synthetic spectra of Piotr Pstr\polhk{a}gowski \cite{Pst-synthetic},
but we prove it with more prosaic methods. We leave the synthetic
approach to the interested reader.
\end{remark}

Before we state the next result, we need
\begin{equation}
\label{hom}
H_{*}K_{2} =
\Gamma(\iota_{2}^{*}) \otimes
\Gamma(z_i^{*}:i>0 ) \otimes  E(u_i^{*}:i\geq 0 ).
\end{equation}

\noindent Here we have $y_{n,j}^{*} = \gamma_{p^{j}}(\iota_{2}^{*})$
dual to $y_{n,j}$ in cohomology.

In Theorem \ref{pd-dual}\ref{pd-dual0}, we compute the $E_{2}$ term
for the {\ASS} for $k(n)_{*}(K_{2})$.  In particular,
$\Gamma(y_{n,1}^{*})$ is there.

%The following is similar to Lemma \ref{lem-E2}.
\crg{
The following is a refined version of Lemma \ref{lem-E2}.  Unfortunately,
the proof of the crucial refinement is intertwined with the proof
of the computation of the ASS, Theorem \ref{pd}, in the next section.
This result seems to best fit this section and it is easy enough
to read off what is needed from Theorem \ref{pd}.
}

\begin{lemma}\label{limita}
\crg{For any prime}, the $z_i$ are all permanent cycles in the {\ASS} for
$k(n)^{*}(K_{2})$ and there is a non-zero differential $d^r (y_{n,i})$ for
some $r \le p^{i}$.  In the {\ASS} for $k(n)_{*}(K_{2})$, $v_{n}^r y_{n,i}^{*}$ is
hit by a differential for some $r \le p^{i}$.
\end{lemma}

\begin{proof}  The image of the map
\begin{displaymath}
BP^{*}(K(\zp,m)) \to H^{*}K(\zp,m)
\end{displaymath}

\noindent is computed by Tamanoi in \cite{Tam:image} (and much
earlier in his 1983 masters thesis in Japan) and then again later in
\cite{RWY}.  In particular, the answer for $m=2$ contains the $z_{i}$,
where $i > 0$.  This map factors through $k(n)^{*}(K_{2})$ so we conclude
that the $z_{i}$ cannot support a differential.

Let $b_{\ell } \in k(n)_{2\ell }\cp$ be the standard generator and
consider the composition
$$
\xymatrix{
\cp \ar[r]^{p} & \cp \ar[r] & K_{2}.
}
$$
Define $b(s) = \sum_{\ell } b_{\ell } s^{\ell }$ and $b_{(i)} =
b_{p^{i}}$.  Note that $b_{(i)}$ maps to $y_{n,i}^{*} \in
k(n)_{*}(K_{2})$.  We follow \cite[Theorem 3.8(ii)]{RW:HR} and use the
fact that for $k(n)$, $[p](s) = v_{n} s^{p^{n}}$.  The composition above
takes $b(s)$ to zero, but the first map takes $b(s) \ra
b(v_{n}s^{p^{n}})$.  In particular, we see that $v_{n}^{p^{i}} b_{(i)}$
maps to zero, giving $v_{n}^{p^{i}} y_{n,i}^{*} = 0 \in
k(n)_{*}(K_{2})$.

Since we must have a $d_{\rho_{n}}(\beta ) = v_{n}^r y_{n,i}^{*}$ with
$r \le p^{i}$, from Theorem \ref{pairing}, we must have a corresponding
$d^\rho_{n}(\alpha ') = v_{n}^r \beta '$ with $|\alpha '| = 2p^{i} =
|y_{n,i}^{*}|$.  
\crg{
We want to show that $y_{n,i} = \alpha'$.
We give the proof for odd primes, $p=2$ requires modifications.
To do this, we use induction on $i$.
We can begin the induction with $i=0$, where $d^1(y_{n,0}) = v_nu_n$, from
(\ref{eq-d1ui}).
We now assume the result for $y_{n,i-1}$. More than that, we assume
Theorem \ref{pd} to the point where we have computed
$d^{\rho_{n}(i-1)}(y_{n,i-1})$ and obtained
$E_{1+\rho_{n} (i-1)}$.
We will be finished if we can show that the only element 
of
$E_{1+\rho_{n} (i-1)}$
that can have a differential on it 
in degree $2p^i$
is $y_{n,i}$.
The only elements that can have differentials are in the $k(n)^*$-free
part of
$E_{1+\rho_{n} (i-1)}$.  In Theorem \ref{pd}, we have $\ell = 2(i-1)$,
so the $k(n)^*$-free part of 
$E_{1+\rho_{n} (i-1)}$
is easy to read off as
\[
E(y_{n,i-1/2})\otimes T_{p^n}(z_{n+i})\otimes P(y_{n,i}) \otimes W_{n,i-1} \otimes
L_n \otimes TZ_{n,i}
\]
The elements of $L_n$ cannot be used because they give us $K(n)^*(K_2)$.
The lowest degree element of $TZ_{n,i}$ is $z_{n+i+1}$ and its degree is
higher than $2p^i$ so we can ignore $TZ_{n,i}$.  The degree of $z_{n+i}$
of $T_{p^n}(z_{n+i})$ is also too high.  All we have left to eliminate
is 
$E(y_{n,i-1/2})\otimes  W_{n,i-1}$. 
The element of lowest degree in $W_{n,i-1}$ is $w_{n,i}$.
For $i \le n$, the degree of $w_{n,i}$ is $2p^{n+i}+1$ 
by (\ref{eq-w})
and too big
to consider.  When $i > n$, we rely on  Lemma \ref{lem-ab}.
Here we have the degree of $w_{n,i}$ is given as
\[
2(c_1 +p^{n+1}c_n) p^{i-n-1} + 2c_n + |w_{n,i-n-1}|.
\]
The term $2p^{n+1}c_n p^{i-n-1} = 2(p^n-1)p^i$, so the degree is too high to
worry about.
All that is left is $y_{n,i-1/2}$, but it has odd degree.  
Although unnecessary at this stage, the degree of $y_{n,i-1/2}$ is
also greater than $2p^i$.
}
\end{proof}

%By the time we get to the point of worrying about
%$y_{n,i}$, it will be the only element in degree $|\alpha '|$ and so
%must be $\alpha '$.

\section{The Adams spectral sequence for odd primes}
\label{thess}
\crb{The $E_{2}$-term of the odd primary Adams spectral sequence for $\k^{*}(K_{2})$ is the subject of Theorem \ref{E2k}.}

\begin{thm}\label{pd}
 {\bf Adams differentials and intermediate terms \crb{for $\k^{*}(K_{2})$}.}

\begin{enumerate}[label={(\roman*)},itemindent=1em]
\item \label{pd-i}
In the odd primary Adams spectral sequence for $\k^{*}(K_{2})$,
the differentials $d^{r}$  for $r\geq 1$ are
\begin{align*}
d^{1} (y_{n,0})
 & = v_{n}u_{n}\\
d^{1} (u_{s})
 & =  \mycases{
v_{n}z_{n-s}^{p^{s}}
      &\mbox{for }0\leq s < n\\
v_{n}z_{s-n}^{p^{n}}
       &\mbox{for }s>2n,} \hspace{-5cm} \\
d^{\rho_{n}(i+1/2)}( y_{n,i+1/2})
 & = v_{n}^{\rho_{n}(i+1/2)} z_{n+i+1} &\mbox{for }i\geq 0\\
\aand
d^{\rho_{n}(i)}(y_{n,i})
 & = v_{n}^{\rho_{n}(i)}w_{n,i}&\mbox{for }i> 0,
\end{align*}

\noindent where $\rho_{n} (\ell /2)$ is as in Lemma \ref{lem-ab}.
\medskip

\item  \label{pd-ii} For each $\ell \geq 0$,
\begin{align*}
E_{1+\rho_{n} (\ell /2)}
  & = E_{\rho_{n} ((\ell +1) /2)}\\
 & =  \bigoplus_{0\leq k \leq \ell }
          \left( S_{n,k /2}\otimes M_{n,k /2}  \right)\\
 &\qquad \oplus \left( k (n)^{*}\otimes
         \left\{\begin{array}{ll}
E (y_{n,(\ell +1)/2})\otimes T_{p^{n}} (z_{n+1+\ell /2})\hspace{-3cm}\\
       &\mbox{for $\ell $ even}\\
T_{p} (y_{n,(\ell +1)/2})\otimes E (w_{n,(\ell +1)/2})\hspace{-3cm}\\
       &\mbox{for $\ell $ odd}\hspace{1.2cm}
               \end{array}\right\}\otimes M_{n,(\ell +1)/2} \right)
\end{align*}

\noindent for  $M_{n,\ell /2}$ and
$S_{n,\ell /2}$ as in (\ref{eq-SM0}).
\end{enumerate}
\end{thm}

In \ref{pd-ii} note that as $\ell $ goes to $\infty $, \crr{both the expressions enclosed
in braces
go to $\Z/p$
and $M_{n,(\ell +1) /2}$  goes to $L_n$.}
Hence the last
summand goes to $ k(n)^{*}\otimes L_n$, and \ref{pd-ii} implies that
the Adams $E_{\infty }$-term is the $k (n)^{*}$-module described in
Theorem \ref{knkn}.

\begin{proof}[Proof of Theorem \ref{knkn} and Theorem
\ref{pd}\ref{pd-ii} assuming Theorem \ref{pd}\ref{pd-i}] The Adams
$E_{1}$-term is
\begin{displaymath}
k (n)^{*}\otimes H^{*}K_{2}.
\end{displaymath}

\noindent  \crbr{Our $d^{1}$ for an odd prime $p$  was computed in \S\ref{subsec-d1}
and our $E_2$ in Theorem \ref{E2k}.
The remaining $k(n)^*$-free part
was $k(n)^*\otimes E(y_{n,1/2})\otimes M_{n,0}$, but
$M_{n,0} = T_{p^n}(z_{n+1})\otimes M_{n,1/2}$,
giving us the answer for the above Adams
$E_2$-term.}

\noindent Higher differentials involve multiplication by higher powers
of $v_{\crr{n}}$, so they cannot affect the torsion submodule.

The remaining Adams differentials, starting with $d^{\rho_{n}(1/2)}
(y_{n,1/2})=v_{n}^{\rho_{n}(1/2)}z_{n+1}$ (where $\rho_{n}
(1/2)=p-1$), have the following effects on the indicated subquotient
rings of $E_{2}$,
\crr{
$E_{1+\rho_{n} (i)}$
and
$E_{1+\rho_{n} (i+1/2)}$
}
for integers $i\geq 0$.
% Defined in ~/math/diary/diary.sty
\begin{numequation}\label{eq-effects}
\begin{split}
y_{n\crr{,}i+1/2}&\mapsto v_{n}^{\rho_{n} (i+1/2)}z_{n+i+1}\\
k (n)^{*}\otimes E (y_{n,i/2})\otimes T_{p^{n}} (z_{n+i+1})
 & \leadsto (k (n)^{*}\otimes E (w_{n,\crbr{n}+i+1}))\\
 &\qquad
      \oplus (T_{\rho_{n} (i+1/2)} (v_{n})
               \otimes z_{n+i+1}T_{p^{n}-1} (z_{n+i+1}))\\
 & = (k (n)^{*}\otimes E (w_{n,\crbr{n}+i+1}))\oplus S_{n,i+1/2}\\
y_{n,i+1}&\mapsto v_{n}^{\rho_{n} (i+1)}w_{n,i+1}\\
k (n)^{*}\otimes T_{p} (y_{n,i+1})\otimes E (w_{n,i+1})
 & \leadsto (k (n)^{*}\otimes E (y_{n,i+3/2}))\\
 &\qquad
      \oplus (T_{\rho_{n} (i+1)} (v_{n})\otimes w_{n,i+1}T_{p-1} (y_{n,i+1}))\\
 & = (k (n)^{*}\otimes E (y_{n,i+3/2}))  \oplus S_{n,i+1}
\end{split}
\end{numequation}%

\noindent These are illustrated by the following diagrams, in which an
arrow $\alpha \to \beta $ labeled by $v_{n}^{r}$ for some $r$ means
that $d^{r}\alpha =v_{n}^{r}\beta $.  Within each of the two diagrams,
all arrows should bear the same label, but all but two labels have been
omitted to avoid clutter.
% Note the use of \hspace commands to position the number correctly.
\begin{numequation}\label{eq-diagrams}
\begin{split}
\hspace{-5mm}
&\xymatrix
@R=10mm
@C=5mm
{
{1}
  &{y_{n,i}}\ar[dl]_(.6){v_{n}^{\rho_{n}(i)}}
    &{y_{n,i}^{2}}\ar[dl]_(.5){}%{v_{n}^{\rho_{n}(i)}}
      &{\dotsb }\ar[dl]_(.5){}%{v_{n}^{\rho_{n}(i)}}
        &{y_{n,i}^{p-2}}\ar[dl]_(.5){}%{v_{n}^{\rho_{n}(i)}}
           &{\hspace{5mm}y_{n,i}^{p-1}\hspace{1cm} }
                  \ar[dl]^(.4){v_{n}^{\rho_{n}(i)}}\\
{w_{n,i}}
  &{w_{n,i}y_{n,i}}
    &{w_{n,i}y_{n,i}^{2}}
      &{\dotsb }
        &{w_{n,i}y_{n,i}^{p-2}}
           &{w_{n,i}y_{n,i}^{p-1}=:y_{n,i+1/2}}
}
\hspace{-1cm}
\\
\hspace{-5mm}
&\xymatrix
@R=4mm
@C=3mm
{
{}\\
{y_{n,i+1/2}}\ar[ddr]_(.4){v_{n}^{\rho_{n}(i+1/2)}}
  &{y_{n,i+1/2}z_{n+i+1}}\ar[ddr]_(.4){}%{v_{n}^{\rho_{n}(i+1/2)}}
    &{\dotsb }\ar[ddr]_(.4){}%{v_{n}^{\rho_{n}(i+1/2)}}
      &{y_{n,i+1/2}z_{n+i+1}^{p^{n}-2}}\ar[ddr]^(.6){v_{n}^{\rho_{n}(i+1/2)}}
         &{y_{n,i+1/2}z_{n+i+1}^{p^{n}-1}=:w_{n,i+n+1}}\\
         % &{y_{n,i+1/2}z_{n+i+1}^{p^{n}-1}}\ar@{=}[r]^(.5){}
         %   &{:w_{n,i+n+1}}\\
{}\\
{1}
  &{z_{n+i+1}}
    &{\dotsb }
      &{z_{n+i+1}^{p^{n}-2}}
         &{z_{n+i+1}^{p^{n}-1}}
}\hspace{-2cm}
\end{split}
\end{numequation}%

\noindent Let $S_{n,i}$ and $S_{n,i+1/2}$ be as in (\ref{eq-SM0}).
Then the new torsion modules created by $d^{\rho_{n} (i)}$ for $i>0$
and $d^{\rho_{n} (i+1/2)}$ for $i\geq 0$ are respectively
\begin{displaymath}
 S_{n,i}\otimes M_{n,i}
\qquad \aand
 S_{n,i+1/2}\otimes M_{n,i+1/2}.
\end{displaymath}

\noindent These give Theorem \ref{pd}\ref{pd-ii}.

This means that the Adams $E_{\infty }$-term has the form indicated in
Theorem \ref{knkn}.  We have to be sure that there are no nontrivial
extensions in $k (n)^{*}$-module structure.

Suppose that for some $i$, $v_{n}^{\rho_{n} (i)}w_{n,i}$ is not zero
but instead has higher filtration than expected.  This would mean
\begin{displaymath}
v_{n}^{\rho_{n} (i)}w_{n,i}=v_{n}^{r_{n,i}}x_{n,i} \qquad
\mbox{for some $x_{n,i}$ with $r_{n,i}>\rho_{n} (i)$.}
\end{displaymath}

\noindent  Then we would have
\begin{displaymath}
v_{n}^{\rho_{n} (i)} (w_{n,i}-v_{n}^{r_{n,i}-\rho_{n} (i)}x_{n,i})=0,
\end{displaymath}

\noindent and we could define
\begin{displaymath}
w'_{n,i}:=w_{n,i}-v_{n}^{r_{n,i}-\rho_{n} (i)}x \qquad \mbox{with }
v_{n}^{\rho_{n} (i)}w'_{n,i}=0
\end{displaymath}
in that filtration of the spectral sequence.
Of course this could also be in a higher filtration,
but this has to end because each degree of Theorem \ref{knkn} is finite.
This follows from $K(n)^*(K_2)$ finite over $K(n)^*$ and $K_2$ being
of finite type.  In the end, our final element
would represent the same element in $E_{\infty }$ as
$w_{n,i}$, so the $k (n)^{*}$-module structure would still be as
stated in Theorem \ref{knkn}.

A similar argument works for the relation $v_{n}^{\rho_{n}
(i+1/2)}z_{n_{i+1}}=0$.
\end{proof}

{\em Overview of the Proof of Theorem \ref{pd}\ref{pd-i}.}
We can assume by induction that we have 
$E_{\rho_{n}(i-1/2)+1}$ 
and we want to get to
$E_{\rho_{n}(i+1/2)+1}$.
There are two parts to the proof.  
We must establish the two differentials, $\rho_n(i)$ and $\rho_n(i+1/2)$,
but at the same time we have to show that there are no other differentials.
The logical way this should go is to show that there are no $d^r$ with
$\rho_n(i-1/2) < r < \rho_n(i)$, then compute $\rho_n(i)$.  After this,
show there are no $d^r$ with $\rho_n(i) < r < \rho_n(i+1/2)$, then compute
$\rho_n(i+1/2)$.  This is not what we do, but it is best to interpret
what we do this way.  It turns out that computing the differentials
and showing there are no other differentials are independent of each other.
Furthermore, it is unnecessary to break up the non-existence of the $d^r$
into two parts because the proof is exactly the same for both parts.
So, what we do is show the two differentials must exist, no matter
if there are other differentials or not.  After that, we show there are
no extra differentials.  The two proofs could go in the opposite order,
or be done in the proper sequence in the way that makes the most sense.  However,
rather than do them in the proper sequence, the reader can interpret the proofs that
way.

%\begin{proof}[Proofs of Theorems \ref{pd}\ref{pd} and \ref{pdh}]
\begin{proof}[Proof of Theorem \ref{pd}\ref{pd-i} ] 
Our proof starts with showing the
asserted differentials must happen.  Then we have to show that there
are no additional differentials.  This is where the full power of {\em
The Pairing} comes in.

% \todo[inline,color=yellow]{7/15/24. {\em The Pairing} is used to
% convert cohomological differentials to homological ones, not to
% exclude unwanted ones.

% 7/25/24. It may be best to prove Theorem \ref{pd} here and worry about
% Theorem \ref{pdh} in the next section.}\bigskip

Assume by induction that we have $E_{\rho_{n}(i-1/2)+1}$. We must have
a $d^{r}(y_{n,i}) = v_{n}^r q$ where $q$ has odd degree and $r \le
p^{i}$ by Lemma \ref{limita}.  There are few odd degree elements in this
range.  We will show that if $q = w_{n,i+1}$, we would have $r > p^{i}$.
This eliminates all $q=w_{n,i+j}$, $j> 1$, because their degree is
even higher.  We want to show
$$
|w_{n,i+1}| -1 -|y_{n,i}| > 2p^{i} (p^{n}-1)
$$
Using the formula for  $w_{n,i+1}$ of (\ref{diff}), the left hand side becomes
\begin{align*}
\lefteqn{|y_{n,i+1}|+1 +2\rho_{n}(i+1)(p^{n}-1) -1 -|y_{n,i}|}
\qquad\qquad\\
 & = 2 (p-1)p^{i}+2\rho_{n}(i+1)(p^{n}-1),
\end{align*}

\noindent which makes the desired inequality
\begin{displaymath}
\rho_{n} (i+1)>\frac{p^{i+n}-p^{i+1}}{p^{n}-1}
     = \frac{p^{i} (p^{n}-p)}{p^{n}-1}
\end{displaymath}

\noindent It is enough to have $\rho_{n}(i+1) > p^{i}$, \crbr{because this 
is larger than the term on the right,} but this is in
 \crr{(\ref{eq-wu}).}

The only remaining elements of odd degree 
\crg{ that have degree less than $|y_{n,i}|+2p^i(p^n-1)$ }
are $y_{n,i}^s w_{n,i}$ for
$s>0$.  However, since we know $w_{n,i} $ meets the {\em Divisibility
Criterion}, we must have $s$ at least $p^{n}-1$.  Then the index of
the differential would be $p^{i} + \rho_{n}(i)$, and this is greater
than $p^{i}$ so can't happen.  We conclude that we {\em must} have
$d^{\rho_{n}(i)}y_{n,i} = v_{n}^{\rho_{n}(i)} w_{n,i}$ as claimed.

Later, when we show there are no extraneous differentials, that
proof actually shows there are no differentials $d^r$ where
$r < \rho_n(i)$, so when we do the computation for this differential,
there is no interference from other possible differentials, because
they do not exist.

Thus, the action of this differential takes place in $k (n)^{*}\otimes  P(y_{n,i})\ot
E(w_{n,i})$ which can be broken up as $k (n)^{*}\otimes P(y_{n,i+1})\ot
T_p(y_{n,i})\ot E(w_{n,i})$.   The
remaining $v_{n}$-torsion free part is $k(n)^{*}\otimes P(y_{n,i+1}) \ot
E(y_{n,i+1/2})$, giving us 
$E_{\rho_{n}(i)+1}$.

By Lemmas \ref{lem-E2} and \ref{limita}, 
we know that
$z_{n+i+1}$ is a permanent cycle and  that
some $v_{n}^r z_{n+i+1}$ must be hit by a differential coming from an
odd degree element.
Remember that we are now working in
$E_{\rho_{n}(i)+1}$.
Furthermore, our proof that there are no other differentials than those
specified shows that there are no differentials $d^r$ with
$\rho_n(i)+1 < r < \rho_n(i+1/2)$.

\begin{lemma}
\label{twice} If $d^{r}(w_{n,i+j})=v_{n}^r z_{n+i+1}$ for some $j >
0$, then $r \le \rho_{n}(i-1/2)$.
\end{lemma}

%\crr{
%\begin{remark}
%Note that the only property we used about $z_{n+i+1}$ was its
%degree.  This will be important later on.
%\end{remark}
%}

\begin{proof}
It is enough to study the $j=1$ case.
\crr{If this differential is too short, then it is even shorter
for $j > 1$.
}
We would have
$$
|w_{n,i+1}| +1 + 2r(p^{n}-1) = |z_{n+i+1}|.
$$
Replace $|w_{n,i+1}|$ using (\ref{diff})
$$
|y_{n,i+1}| + 1 + 2\rho_{n}(i+1)(p^{n}-1) +1 + 2r(p^{n}-1) = |z_{n+i+1}|.
$$
Plugging in the numbers for $y_{n,i+1}$ and $z_{n+i+1}$ and rearranging, we get
$$
  2\rho_{n}(i+1)(p^{n}-1) + 2r(p^{n}-1) = 2p^{n+i+1}  - 2p^{i+1} = 2p^{i+1}(p^{n}-1).
$$
So, $r = p^{i+1} -\rho_{n}(i+1)$.

We need to show this is $\le \rho_{n}(i-1/2)$.
\crr{
We use the formulas from
Lemma \ref{lem-ab}.
We need to show that $p^{i+1} \le \rho_n(i+1) + \rho_n(i-1/2)$.
This is easy for small $i$, so using the formulas and induction,
we need to show
\[
p^{i+n+2} \le \rho_n(i+n+2)+\rho_n(i+n+1/2).
\]
The right hand side is
\[
p^{i+2}(p^n-1) +1 + \rho_n(i+1) + p^{n+i}(p-1)-1 +\rho_n(i-1/2).
\]
Expanding and using induction, this is greater than or equal to
\[
p^{n+i+2} -p^{i+2} + p^{n+i+1} - p^{n+i} + p^{i+1}.
\]
For this to be greater than or equal to $p^{n+i+2}$, we need
\[
p^{n+i+1} +p^{i+1} \ge p^{n+i} + p^{i+2}.
\]
This is obvious for $n > 1$, and we get an equality when $n=1$.
}
\end{proof}

Lemma \ref{twice} rules out all $w_{n,i+j}$ with $j > 0$ as the source
of a differential hitting a $v_{n}$-multiple of $z_{n+i+1}$ because
we assume that $E_{1+\rho_n(i)}$  has already been computed.

Now the only odd degree elements left 
\crg{in degrees less than $|z_{n+i+1}|$}
are the $y_{n,i+1}^s
y_{n,i+1/2}$.  We know $y_{n,i+1/2}$ would work with differential
$\rho_{n}(i+1/2)$ because of the {\em Divisibility Criterion}, Proposition \ref{limit}
and (\ref{diff}).  The
{\em Divisibility Criterion } requires $s $ to be a multiple of $p^{n}
-1$.  The lowest non-zero $s$ is $s=p^{n}-1$ and this would give a
differential of length $\rho_{n}(i+1/2) - p^{i+1}$, but $p^{i+1} >
\rho_{n}(i+1/2)$ so this cannot happen.  We {\em must} have
$d^{\rho_{n}(i+1/2)} ( y_{n,i+1/2}) = v_{n}^{\rho_{n}(i+1/2)}
z_{n+i+1}$.

The part of $E_{\rho_n(i)+1}$ that
the action of $d^{\rho_{n}(i+1/2)}$
takes place in is $ P(v_{n}) \ot E(y_{n,i+1/2}) \ot T_{p^{n}}(z_{n+i+1}) $
and results in the $P(v_{n})$-free part being $E(w_{n,i+(n+1)})$,
giving us \linebreak $E_{\rho_{n}(i+1/2)+1}$.

%\todo[inline,color=yellow]{8/14/24. The next paragraph is about
%homology.}\bigskip

Having computed these differentials, we can use
{\em The Pairing} of Theorem \ref{pairing} to get the dual
differentials for the {\ASS} for $k(n)_{*}(K_{2})$ in Theorem
\ref{pd-dual}.  We first show $d_{\rho_{n}(i)}(w^{*}_{n+i}) =
v_{n}^{\rho_{n}(i)} y^{*}_i$.  We know, from Lemma \ref{limita} that
some differential must hit some $v_{n}^r y^{*}_i$ with $r \le p^{i}$.
From {\em The Pairing}, we know that some element, $q$, in the degree
of $y^{*}_i$ must have $d_{\rho_{n}(i)}(m) = v_{n}^{\rho_{n}(i)} q$.
However, in $E_{\rho_{n}(i-1/2)}$, we see that $y^{*}_i$ is the only
element there is in that degree and $w^{*}_{n+i}$ is the only odd
degree generator in the correct degree.  The only option is the
expected result.  Again, the pairing gives us a $d_{\rho_{n}(i+1/2)}$
in degrees corresponding to $z^{*}_{n+i+1}$ and $w^{*}_{n+i+1/2}$.
There are no other options, so $d_{\rho_{n}(i+1/2)}$ is as advertised.

Having computed these two must-have differentials, it gives us the
description of $E_{\rho_{n}(i+1/2)+1}$ of Theorem \ref{pd}\ref{pd-ii}.
% \begin{equation}
% E_{\rho_{n}(i+1/2)+1}=
% k(n)^{*}\otimes P(y_{n,i+1})
% \otimes
% E(w_{n,i+i + 1}:0\le i\le n)
% \otimes
% T_{p^{n}}(z_{n+i+s+2}:s\geq 0)
% \end{equation}

%\todo[inline,color=yellow]{8/14/24. Stopped here today.}\bigskip

We are not finished. {\em We must show that there are no
extraneous $d^r$s.}  We assume, by induction, that we have
$E_{\rho_{n}(i-1/2)+1}$.  
The first step is to show there are no differentials with
$\rho_{n}(i-1/2)+1< r < \rho_{n}(i)$.  Any such differential would
take place on  
$E_{\rho_{n}(i-1/2)+1}$.  
There are no differentials on the $z$'s because they are permanent
cycles.  The element $y_{n,i}$ is reserved for our special 
differentials as is $w_{n,i}$.  The only possible differentials are on the
$w_{n,i+j}$ with $1 \le j \le n$.  If we show there are no such
differentials, we get our $ \rho_n(i)$.  
Then we have to consider $d^r$ with $\rho_{n}(i)+1 < r < \rho_n(i+1/2)$.  
These would take place on $E_{\rho_n(i)+1}$.  Here we cannot use
$y_{n,i+1/2}$ or $z_{n+1+i}$ because they are reserved for the special
differentials already found to be necessary.  The remaining options
are $y_{n,i+1}$ and 
$w_{n,i+j}$ with $1 \le j \le n$. But this is the same as we had before
with the exception of $y_{n,i+1}$.  This is easy to eliminate because the lowest
odd degree element is $w_{n,i+1}$.  From our computation of the special
differentials, we know this would require (from (\ref{diff})) $\rho_n(i+1) > \rho_n(i+1/2)$, so
the differential would be too long.  We can now concentrate on showing
there are no extra differentials $d^r$ with $\rho_n(i-1/2) < r \le \rho_n(i+1/2)$
that start on a $w_{n,i+j}$, $1 \le j \le n$.

Let $r$ be the smallest $r$ in the range
$\rho_{n}(i-1/2) < r \le \rho_{n}(i+1/2)$ with $d^r (w_{n,i+j}) =
v_{n}^r \beta \ne 0$, for the $w_{n,i+j}$ of smallest degree.  We
know from Theorem \ref{pairing}, {\em The Pairing}, that there is a
$\beta'$, with $|\beta'| = |\beta|$, in the homology {\ASS} with
$d_r(\beta ')= v_{n}^r \alpha' \ne 0$.  If $\beta'$ is decomposable,
then there must be an element, $\beta''$, with lower degree than
$\beta'$ with $d_r(\beta'') = v_{n}^r \alpha'' \ne 0$.  For
example, if $\beta' = ab$, then $d_r(\beta ') = d_r(a)b\pm ad_r(b)$
and either $d_r(a)$ or $d_r(b)$ is non-zero.  In either case, we get
our $\beta''$ with degree less than $|\beta'|$.  Again, by {\em The
Pairing}, there is an $\alpha$ in the cohomology {\ASS} with $|\alpha
|=|\alpha''| < |w_{n,i+j}|$ with $d^r(\alpha ) \ne 0$.  This
contradicts our choice of $w_{n,i+j}$.  We conclude that if there is
such an $r$, $\beta'$ is indecomposable.  Theorem \ref{pairing}, {\em
The Pairing}, is pretty vague about what the corresponding elements
are.  All it really gives us are degrees.

Since we started with the odd degree
$w_{n,i+j}$, we are looking for an even degree target element.
However, we know where all the even degree  indecomposables
are in the homology spectral sequence,  dual to 
$E_{\rho_n(i-1/2)+1}$.
These elements are the
$y^{*}_{s}$, $s \ge i$, and the $\gamma_{p^k}(z^{*}_{n+i+s})$, with $k < n$ and
$s > 0$.  .  We have
similar looking elements in 
$E_{\rho_n(i-1/2)+1}$,
namely, $y_{n,s}$, $s \ge i$,  and $z_{n+i+s}^{p^k}$, $s > 0$, $ k < n$.
They are not known to be "dual" in any sense, but they are in the
right degrees.  All we will use about these cohomology elements is
their degree.  If we can show that there are no differentials that hit
elements in these degrees, we are done.
We overlooked some elements in our original proof, but a very persistent referee
forced us to find them.  This led to a complete reworking of the proof,
a dramatic improvement.

We have three main ways to show a differential cannot exist.  (1) We can use
the {\em Divisibility Criterion}, (2) we can show that a prospective $d^r$
has $ r > \rho_{n}(i+1/2)$, or (3) we can show that $r \le \rho_{n}(i-1/2)$.

First we have to check to see if there is some $s$ with
$d^r (w_{n,i+j}) = v_{n}^r y_{n,s}$, again, we repeat, only using the
degree of $y_{n,s}$.  
(If we could actually use $y_{n,s}$ this would be easy because we know
it is a source and cannot be a target.)
For this, we must have
$$
|w_{n,i+j}|+1+2r(p^{n}-1)= |y_{n,s}|
$$
but we can replace the first term using (\ref{diff})
$$
|y_{n,i+j}|+1+2\rho_n(i+j)(p^{n}-1)
+1+2r(p^{n}-1)= |y_{n,s}|
$$
so
$$
 |y_{n,s}| -|y_{n,i+j}| -2 = 2p^s - 2p^{i+j} -2
$$
is both positive and divisible by $2(p^{n}-1)$.  
This cannot be zero mod $2(p^{n}-1)$ so the {\em Divisibility Criterion}
tells us we cannot have this differential.

The elements $z_{n+i+s}^{p^k}$ below are in degrees that correspond to
the degrees of the remaining even degree generators in the homology version.
We have to show, using only their degrees, that there is {\bf no} differential
\begin{align}
\begin{split}
\label{allcontorted}
d^r ( w_{n,i+j})& = v_{n}^r z_{n+i+s}^{p^k}
\quad \text{with} \\
0 <  j \le n, \   0 < s, &  \  0 \le k < n, \ 0 <  i  \\
   \text{and}\
\rho_{n}(i-1/2)& <   r \le \rho_{n}(i+1/2).
\end{split}
\end{align}

We have
$$
|w_{n,i+j}| + 1 +  2r(p^n-1) = | z_{n+i+s}^{p^k}|
$$
We replace $|w_{n,i+j}|$ with $|y_{n,i+j}|+1+2\rho_{n}(i+j)(p^{n}-1)$
from (\ref{diff})
so we have

$$
|y_{n,i+j}|+1+2\rho_{n}(i+j)(p^{n}-1)
+1+2r(p^{n}-1) =
| z_{n+i+s}^{p^k}|.
$$
Turning this into numbers and rearranging,
\begin{align*}
2r(p^n-1)& = | z_{n+i+s}^{p^k}| -2 
-|y_{n,i+j}|
-2\rho_{n}(i+j)(p^{n}-1) \\
&=2 p^{n+i+k+s} + 2p^k -2 - 2p^{i+j} 
-2\rho_{n}(i+j)(p^{n}-1)
\end{align*}

First we ask, when is this too big, that is, when is
$$
2\rho_n(i+1/2)(p^n-1) < 2p^{n+i+k+s} + 2p^k -2 - 2p^{i+j} -2\rho_{n}(i+j)(p^{n}-1)
$$
Rearranging, when is
$$
2\rho_n(i+1/2)(p^n-1) +2 + 2p^{i+j} +2\rho_{n}(i+j)(p^{n}-1) < 2p^{n+i+k+s} + 2p^k 
$$
From (\ref{eq-wu}) we know $\rho_n(i+j) \le p^{i+j}$, so when is
$$
2\rho_n(i+1/2)(p^n-1) +2 + 2p^{i+j} +2p^{i+j}(p^{n}-1) < 2p^{n+i+k+s} + 2p^k.
$$
The large terms on left and right are the $2p^{n+i+j}$ and $2p^{n+i+k+s}$.
So this inequality holds when $n+i+j < n+i+k+s$,  $j < k+s$, or $s > j -k$.

What we have left is  $s \le j-k$, or $s = j-k -t$ with $t \ge 0$.  We
can't have $t$ too big because $j \le n$.  Using the above approach, it is
easy to see that the differential is too short if
$j > k+s+1$, or $s < j-k -1$.  Unfortunately, that misses a couple of cases,
namely $s=j-k -\epsilon$ when $\epsilon$ is 0 or 1.  Those cases are
 more delicate, but since they  are also too short, we do them all at once.

When our differential is too short, we have
$$
2\rho_n(i-1/2)(p^n-1) \ge  2p^{n+i+k+s} + 2p^k -2 - 2p^{i+j} -2\rho_{n}(i+j)(p^{n}-1)
$$
Substitute $s=j-k-t$ to get
$$
2\rho_n(i-1/2)(p^n-1) \ge  2p^{n+i+j-t} + 2p^k -2 - 2p^{i+j} -2\rho_{n}(i+j)(p^{n}-1)
$$
When $i+j < n+1$ we can compute all the numbers and show this is true easily, so
we will assume $i+j \ge n+1$.
Rearrange and use Lemmas \ref{lem-f} and \ref{lem-ab} to get
\begin{align*}
2\rho_n(i-1/2)(p^n-1) 
&+2 + 2p^{i+j}  \\
+2\Big(p(p^n-1)p^{i+j-n-1} +& \rho_n(i+j-n-1) + 1\Big)
(p^{n}-1) \\
 \ge  
2p^{n+i+j-t}& + 2p^k 
\end{align*}
Replace $\rho_n(i-1/2) $ with $p^i-\rho_n(i-1)$ and multiply everything out and rearrange to 
get the left hand side as
$$
2p^{n+i}+2\rho_n(i-1) +2 + 2p^{n+i+j} + 2p^{i+j-n} + 2\rho_n(i+j-n-1)p^n +2p^n
$$
and the right hand side
$$
2p^i+ 2\rho_n(i-1)p^n + 2p^{i+j} + 2\rho_n(i+j-n-1) +2p^{n+i+j-t}+2p^k.
$$
Because we know $\rho_n(i-1) \le p^{i-1}$, $k < n$, and $\rho_n(i+j-n-1) \le p^{i+j-n-1}$,
the largest term on the left is the $2p^{n+i+j}$ and the inequality holds
when $t > 0$.  When $t=0$ we can cancel the big terms on both sides and we
have the left side is
$$
2p^{n+i}+2\rho_n(i-1) +2 +  2p^{i+j-n} + 2\rho_n(i+j-n-1)p^n +2p^n
$$
and the right
$$
2p^i+ 2\rho_n(i-1)p^n + 2p^{i+j} + 2\rho_n(i+j-n-1)  +2p^k.
$$
The largest term on the left is now $2p^{n+i}$.  On the right, the largest is
$2p^{i+j} $ when $j=n$.  When this happens those two terms cancel and we are
looking at
$$
2\rho_n(i-1) +2 +  2p^{i} + 2\rho_n(i-1)p^n +2p^n
\ge 
2p^i+ 2\rho_n(i-1)p^n  + 2\rho_n(i-1) +2p^k
$$
where lots of things cancel to give us
$$
2 +2p^n
\ge 
2p^k
$$
which is true because $k < n$.

This concludes all of the cases we needed to check.
There are no more differentials than those already produced.

Of course if there are no more differentials in the {\ASS} for
$k(n)^{*}(K_{2})$, then {\em The Pairing } says there are no more for
$k(n)_{*}(K_{2})$.
\end{proof}

%\section{The Dual of everything}
\crb{\section{From cohomology to homology}
\label{alldual}
}

\crb{We now turn from $k(n)^{*}(K_{2})$ to $k(n)_{*}(K_{2})$.} We have
already given $H_{*}K_{2}$ in (\ref{hom}).  We need the $Q_n$ homology
of $H_{*}K_{2}$, but it is just dual to the $Q_n$ homology of
$H^{*}K_{2}$ described in Section \ref{qhomsec}.  It gives us the
$E_{2}$ term of the {\ASS} for $k(n)_{*}(K_{2})$,
\crb{which we spell out in Theorem \ref{pd-dual}\ref{pd-dual0}.}

We give the $E_{2}$ term of the {\ASS} for $k(n)_{*}(K_{2})$ and
describe all the differentials \crb{in Theorem \ref{pd-dual}}, and give the
final result as a $k(n)_{*}$-module  \crb{in Theorem \ref{knkn-dual}}.  The
proofs are dual to the proofs for $k(n)^{*}(K_{2})$.
% and are actually,
% of necessity, carried out simultaneously with those.

%\todo[inline,color=yellow]{7/13/24. This section needs more work.}\bigskip

\crb{Using the notation of (\ref{eq-hom-case}), we have elements
% Defined in ~/math/diary/diary.sty
\begin{numequation}\label{eq-dual-gens}
\begin{split}
%\begin{array}{rlrl}
v_{n} &\in G_{2}^{2(p^{n}-1),1}=E_{2}^{1,2(p^{n}-1)+1},\\
y_{n,j}^{*}
      &\in G_{2}^{2p^{j},0}=E_{2}^{0,2p^{j}},\\
w_{n,i}^{*}
      &\in G_{2}^{2p^{n+i}+1,0}=E_{2}^{0,2p^{n+i}+1}\\
\aand
z_{j,s}^{*}=\gamma_{p^{s}} (z_{j}^{*})
      &\in G_{2}^{2p^{s}(p^{j}+1),0}=E_{2}^{0,2p^{s}(p^{j}+1)}.
%\end{array}
\end{split}
\end{numequation}%

The dual analog of (\ref{eq-convenient}) is
% Defined in ~/math/diary/diary.sty
\begin{numequation}\label{eq-convenient-dual}
\begin{split}
EE_{n}^{*}
 & := E (u_{s}^{*}:0\leq s<n)\otimes E (u_{2n+s}^{*}: s>0) ,\\
W_{n,i}^{*}
 & :=E (w_{n,i+s}^{*}:1\leq s\leq n), \\
L_{n}^{*}
 &: = \bigotimes_{0<s <n}\Gamma _{p^{n-s }} (z_{s }^{*}),\\
TZ_{n,i}^{*}
 & := \Gamma _{p^{n}} (z_{n+s}^{*}:s>i),\qquad \mbox{  and }\\
PZ_{n} &:= \Gamma  (\gamma_{e_{n} (s)} (z_{s}^{*}):s>0)
\qquad \mbox{with $e_{n} (s)$ as in (\ref{eq-convenient}).}
\end{split}
\end{numequation}%

We also need dual analogs of (\ref{onlyd1}) and (\ref{eq-SM0}).
\begin{numequation}\label{eq-SM0-dual}
\begin{split}
D_{1}^{*}
 &\coloneqq  \Gamma _{p} (y_{n,0}^{*})
  \otimes E(w_{n,0}^{*})
\otimes EE_{n}^{*}
    \otimes PZ_{n}^{*},\\
S_{n,0}^{*}
 &:= D_{1}^{*}, \\
M_{n,i}^{*}
 &: = \Gamma  (y_{n,i+1}^{*})\otimes W_{n,i}^{*}
          \otimes L_{n}^{*}\otimes TZ_{n,i}^{*}
\qquad \mbox{for $i\geq 0$},\\
M_{n,i+1/2}^{*}
 & :=\Gamma  (y_{n,i+1}^{*})\otimes W_{n,i}^{*}
           \otimes L_{n}^{*}\otimes  TZ_{n,i+1}^{*}
        \quad \mbox{for }i\geq 0 ,\\
S_{n,i}^{*}
 &:= T_{\rho_{n}(i)}(v_{n})\ot y_{n,i}^{*}
             \overline{\Gamma _{p-1}(y_{n,i}^{*})} ,
       \\
 & \phantom{:} = T_{\rho_{n}(i)}(v_{n})\ot
\left\{\gamma_{s}(y_{n,i}^{*}):1\leq s\leq p-1 \right\}
  \quad \mbox{for }i>0 ,\\
\aand
S_{n,i+1/2}^{*}
 &:= T_{\rho_{n}(i+1/2)}(v_{n})
       \ot z_{n+i+1}^{*} \overline{\Gamma _{p^{n}-1}(z_{n+i+1}^{*})} \\
 &\phantom{:} = T_{\rho_{n}(i+1/2)}(v_{n})\ot
 \left\{\gamma_{s}(z_{n+i+1}^{*}):1\leq s\leq  p^{n}-1 \right\}
         \quad \mbox{for }i\geq 0.
\end{split}
\end{numequation}%

Using this notation, we can now state the dual of Theorems \ref{E2k}
and \ref{pd}.\medskip

\begin{thm}\label{pd-dual}
 {\bf The Adams $E_{2}$-term, differentials and intermediate terms
for\linebreak $\k_{*}(K_{2})$.}

\begin{enumerate}[label={(\roman*)},itemindent=1em]
\item \label{pd-dual0}
The $E_{2}$ term of the odd primary Adams
spectral sequence for $\k_{*}(K_{2})$ is
\begin{displaymath}
P(v_{n})\otimes L_{n}^{*}
\ot \Gamma (y_{n,1}^{*})
\ot
E(y_{n,1/2}^{*})
\otimes W_{n,0}^{*}
\otimes  TZ_{n,0}^{*}
\end{displaymath}

\noindent plus a computable family of filtration-0 $\Z/p$'s
annihilated by $v_{n}$ coming from the dual of the $\E$-free part of
$H_{*}K_{2}$, specified in Theorem \ref{knkn}.

\bigskip
\item \label{pd-duali}
In the odd primary Adams spectral sequence for $\k_{*}(K_{2})$,
the differentials $d_{r}$  for $r\geq 1$ are
\begin{align*}
d_{1} (u_{n}^{*})
 & = v_{n}y_{n,0}^{*},\\
d_{1} (\gamma_{p^{s}} (z_{n-s}^{*}))
 & = v_{n}u_{s}^{*} &&\mbox{for }0\leq s < n,\\
d_{1} (\gamma_{p^{n}} (z_{s-n}^{*}))
 & = v_{n}u_{s}^{*} &&\mbox{for }s>2n,\\
d_{\rho_{n}(i+1/2)} (z_{n+i+1}^{*})
 & = v_{n}^{\rho_{n}(i+1/2)}  y_{n,i+1/2}^{*}
                    &&\mbox{for }i\geq 0,\\
\aand
d_{\rho_{n}(i)}(w_{n,i}^{*})
 & = v_{n}^{\rho_{n}(i)}y_{n,i}^{*}
                    &&\mbox{for }i> 0.
\end{align*}

\bigskip

\item  \label{pd-dualii} For each $\ell \geq 0$,
\begin{align*}
E_{1+\rho_{n} (\ell /2)}
  & = E_{\rho_{n} ((\ell +1) /2)}\\
 & =  \bigoplus_{0\leq k \leq \ell }
          \left( S_{n,k /2}^{*}\otimes M_{n,k /2}^{*}  \right)\\
 &\qquad \oplus \left( k (n)^{*}\otimes
         \left\{\begin{array}{ll}
E (y_{n,(\ell +1)/2}^{*})
    \otimes \Gamma _{p^{n}} (z_{n+1+\ell /2}^{*})\hspace{-3cm}\\
       &\mbox{for $\ell $ even}\\
\Gamma _{p} (y_{n,(\ell +1)/2}^{*})
    \otimes E (w_{n,(\ell +1)/2}^{*})\hspace{-3cm}\\
       &\mbox{for $\ell $ odd}\hspace{1.2cm}
               \end{array}\right\}\otimes M_{n,(\ell +1)/2}^{*} \right)
\end{align*}

\noindent for  $M_{n,\ell /2}^{*}$ and
$S_{n,\ell /2}^{*}$ as in (\ref{eq-SM0-dual}).
\end{enumerate}
\end{thm}

The dual of Theorem \ref{knkn} is

\begin{thm}\label{knkn-dual}
For an odd prime $p$,   $\k_{*}(K_{2})$
has the following three summands as a $k(n)_{*}$-module:
\begin{enumerate}[label={(\roman*)},itemindent=1em]

\item \label{knkn-duali} The {\em $k (n)_{*}$ free summand},
$k(n)^{*}\otimes L_{n}^{*}$, for $L_{n}$ as in (\ref{eq-convenient-dual}).

\item \label{knkn-dualii}  The {\em higher torsion summand},
\begin{displaymath}
 \bigoplus_{\ell >0}\left( M_{n,\ell /2}^{*}
        \otimes  S_{n,\ell /2}^{*} \right),
\end{displaymath}

\noindent for $M_{n,\ell /2}^{*}$ and $S_{n,\ell /2}^{*}$ as in
(\ref{eq-SM0-dual}).

\item \label{knkn-dualiii} The {\em elementary torsion summand},
$S_{n,0}^{*}\otimes M_{n,0}^{*}$ as in (\ref{eq-SM0-dual}).

\end{enumerate}
\end{thm}

These follow from Theorems \ref{pd} and \ref{knkn} using the duality
of Theorem \ref{pairing}.  }

%\todo[inline,color=yellow]{1/7/25. Stopped here today.}\bigskip

\section{Modifications for $p=2$}
\label{mod2}

All we do in this section is to lay out the results for $k(n)^{*}(K_{2})$
for $p=2$.  We skip the homology version and proofs.  We do this
with a twinge of guilt.  The very first case done was the
$p=2$, $n=1$ case, and there, the generally useful {\em Divisibility
Criterion } is worthless.  Consequently, there are lots of little
{\em ad hoc} arguments that must be done in that case.

For $p=2$, $H^{*}K_{2} = P(\iota_{2})\otimes_{i\ge 0} P(u_i)$, with $u_i =
Q_i \iota_{2}$.  We let $u_i^2 = z_{i+1} = Q_{i+1}Q_0 \iota_{2}$
($z_0=0$).  In an attempt to be as similar as possible with notation,
we have
\begin{align*}
| u_i |
 &  = 2\times 2^{i}+1\\
\aand | z_i |
 & = 2(2^{i} +1)\qquad \mbox{for }i>0\\
| Q_n |
 &  = 2\times 2^{n} -1 \\
y_{n,j}
  = \iota_{2}^{2^{j}}\qquad \mbox{for $j\geq 0$ in degree $2 \times 2^{j}$}
\end{align*}

\noindent  with
\begin{align*}
Q_{n}u_{i}
 & = \mycases{
(u_{n-i-1})^{2^{i+1}}
    = z_{n-i}^{2^{i}}
       &\mbox{for }0 \le i <  n\\
0      &\mbox{for }n\leq i\leq 2n\\
(u_{i-n-1})^{2^{n+1}} = z_{i-n}^{2^{n}}
       &\mbox{for } i>2n
}\\
Q_n y_{n,0}
 &   =   u_n
\end{align*}

\noindent
The formulas used for odd primes mostly work here
\begin{align*}
w_{n,i}
 & := \mycases{
u_{n}
       &\mbox{for }i=0\\
u_{n+i} + u_{n-i}  z_{i}^{2^{n} - 2^{n-i}}
       &\mbox{for }0 < i \le n\\
%u_{2n+1}+  y_{n,0} w_{n,0}z_{n+1}^{2^{n}-1}
u_{2n+1}+  y_{n,1/2} z_{n+1}^{2^{n}-1}
       &\mbox{for }i=n+1\\
%y_{n,i-n-1} w_{n,i-n-1}
y_{n,i-n-1/2}z_{i}^{2^{n}-1}
       &\mbox{for }i>n+1
}\\
y_{n,j+1/2}
 &: = y_{n,j} w_{n,j}.
\end{align*}

\noindent The only difference between the above and (\ref{eq-w}) is the definition of $w_{n,n+1}$, which here includes $u_{2n+1}$.  For odd primes we have
\begin{displaymath}
d^{1}u_{n+1} = v_{n}z_{n+1}^{p^{n}},
\qquad \aand
d^{p-1}y_{n,1/2}=v_{n}^{p-1}z_{n+1}
\end{displaymath}

\noindent making the heuristic expression
\begin{displaymath}
w_{n,n+1}:=y_{n,1/2}z_{n+1}^{p^{n}-1}-v_{n}^{p-2}u_{n+1}
\end{displaymath}

\noindent a cycle.  We do not see the second term for $p>2$ because it
has higher filtration, but for $p=2$ both terms have filtration 0.
%\end{displaymath}

\noindent
% \noindent
% $$
% w_{n,i} = u_{n+i} + u_{n-i} z_{i}^{2^{n} - 2^{n-i}}
% \quad 0 \le i \le n
% \qquad
% y_{n,j+1/2} = y_{n,j} w_{n,j}
% $$
% $$
% w_{n,(n+1)} = u_{2n+1}+  y_{n,0} w_{n}z_{n+1}^{2^{n}-1}
% \quad
% w_{n,j+(n+1)} =  y_{n,j} w_{n,j}z_{n+j+1}^{2^{n}-1}
% \quad
% j > 0
% $$

% In the {\ASS} for odd primes, we had  $d^1(y_{n,0}) = v_{n}w_n$ so that both
% $y_{n,0}$ and $w_n$ were part of the $\E$-free part of $E_{2}$.
% In the $p=2$ case we have $d^1(y_{n,0}w_n) = v_{n}w_n^2 = v_{n}z_{n+1}$
% so  $z_{n+1}$, and all its powers, are also in the $\E$-free part. In general, we have
% $d^1(y_{n,0} w_n^k) = v_{n}w_n^{k+1}$.
% Because we would normally have $w_{n,(n+1)} = y_{n,0} w_n z_{n+1}^{2^{n}-1}= y_{n,0} w_n^{2^{n+1}-1}$,
% giving $d^1(w_{n,(n+1)}) = v_{n}z_{n+1}^{2^{n}}$, we would not have $w_{2n+1}$.
% Making the adjustment above allows us to keep a new version of $w_{2n+1}$
% with $d^1(w_{2n+1}) = Q_n(w_{2n+1})=0$.

We can compute the $Q_{n}$-homology of $H^{*}K_{2}$ with a diagram like that of (\ref{eq-bigraded}).

% Rewriting $H^{*}K_{2}$ at $p=2$
% \begin{displaymath}
% \begin{array}[]{c}
% \Big(E(y_{n,0}) \ot P(w_n)\Big)
%  \otimes \Big(E(u_i:0\le i < n){\ot} P(z_{n-i}^{2^{i}}:0\le i < n)\Big)\\
% \otimes \Big(E(u_{n+(n+1)+i}:i>0) {\ot} P((z_{n+1+i})^{2^{n}}:i>0)\Big)\\
% \otimes \bigotimes_{0<i<n}T_{2^{i}}z_{n-i}
% \end{array}
% \end{displaymath}

% \noindent
% $$
% \Big(E(y_{n,0}) \ot P(w_n)\Big)
% \ot \Big(E(u_i)\underset{0\le i < n}{\ot} P(z_{n-i}^{2^{i}})\Big)
% \ot \Big(E(u_{n+(n+1)+i}) \underset{0 < i}{\ot} P((z_{n+1+i})^{2^{n}})\Big)
% $$
% $$
% \underset{0 < i < n}{\ot}T_{2^{i}}z_{n-i}
% \ot
% P(y_{n,1})
% \underset{0 \le i \le  n}{\ot}T_{2^{n+1}}(w_{n,i+1})
% \underset{0 < s}{\ot} T_{2^{n}}(z_{n+(n+1)+s+1})
% $$

% The  $Q_n$ homology  is
% $$
% \underset{0 < i < n}{\ot}T_{2^{i}}z_{n-i}
% \ot
% P(y_{n,1})
% \underset{0 \le i \le  n}{\ot}T_{2^{n+1}}(w_{n,i+1})
% \underset{0 < s}{\ot} T_{2^{n}}(z_{n+(n+1)+s+1})
% $$

\begin{thm}\label{E2k2}
We have elements $v_{n}\in G^2_{-2(2^{n}-1),1}$, $y_{n,1} \in G^2_{4,0}$,
$w_{n,i} \in G^2_{2^{n+i+1}+1,0}$,
and $z_{j} \in G^2_{2^{j+1}+2,0}$.
The $E_{2}$ term of the $p=2$ Adams spectral sequence for $\k^{*}(K_{2})$ is
\begin{displaymath}
\begin{array}[]{c}
P (v_{n})\otimes
\displaystyle{\bigotimes_{0<i<n}T_{2^{i}}z_{n-i} \otimes P (y_{n,1})}
\\
{\ot} T_{2^{n+1}}(w_{n,i+1}:0 \le i \le n )
\otimes T_{2^{n}}( z_{2n+1+i+1} :i>0 )
\end{array}
\end{displaymath}

\noindent plus a computable family of filtration-0 $\zt$'s annihilated
by $v_{n}$ coming from the $\E$-free part of $H^{*}K_{2}$.
\end{thm}

For convenience we reset $z_{n+i+1} = w_{n,i}^2$ for $0 < i \le n+1$.

\begin{prop}
\label{pd2}
For $p=2$, the differentials in the $p=2$ Adams spectral
sequence for $\k^{*}(K_{2})$ are:
\begin{enumerate}[label={(\roman*)},itemindent=0em]
\item \label{pd2i}
For $0 < j \le  n+1$, $\rho_{n}(j) = 2^{j} = \rho_{n}(j+1/2)$.  Although $y_{n,j+1/2} = y_{n,j} w_{n,j}$, for $j \le n+1$, this is not a generator.
\begin{align*}
d^{2^{j}}(y_{n,j})
 & = v_{n}^{2^{j}} w_{n,j} \\
\aand
d^{2^{j}}(y_{n,j} w_{n,j})
 &  =
v_{n}^{2^{j}} w_{n,j}^2
=
v_{n}^{2^{j}} z_{n+j+1}.
\end{align*}

\item \label{pd2ii}

For $j>n+1$,
\begin{align*}
d^{\rho_{n}(j)}(y_{n,j})
 & = v_{n}^{\rho_{n}(j)}w_{n,j}   \\
\aand
d^{\rho_{n}(j+1/2)}(y_{n,j+1/2})
 &  = v_{n}^{\rho_{n}(j+1/2)} z_{n+j+1}.
\end{align*}

\item \label{pd2iii}
 For $0 <  j \le  n+1$, $\rho_{n}(j)=2^{j}=\rho_{n}(j+1/2)$.
Ignoring the permanent
free terms and the previously created $v_{n}$-torsion,
\begin{align*}
E_{2^{j}+1}
 & = k(n)^{*}\otimes P(y_{n,j+1})
     \otimes \underset{j \le i \le n}{\bigotimes } T_{2^{n+1}}(w_{n,i+1}) \\
 &\qquad
    \otimes  \underset{0 \le  i < j}{\bigotimes}E(w_{2n+2+i})
   \otimes  \underset{0 <  s}{\bigotimes} T_{2^{n}}(z_{2n+2+s}).
\end{align*}

\noindent

\item \label{pd2iv}
For $n+1 < j$,
\begin{align*}
E_{\rho_{n}(j)+1}
 & =  k(n)^{*}\otimes P(y_{n,j+1})
\otimes
E(y_{n,j+1/2})\\
 & \qquad
\otimes  \underset{0<i\le n}{\bigotimes} E(w_{n,j+i})
 \otimes  \underset{0 \le s}{\bigotimes} T_{2^{n}}(z_{n+j+s+1})
\end{align*}
\end{enumerate}

\end{prop}

We could rewrite $T_{2^{n+1}}(w_{n,i+1})$ as $E(w_{n,i+1})\ot
T_{2^{n}}(z_{n+i+2})$ for $0 \le i \le n$.  If we did that, we could
write proposition \ref{pd2} without the exceptional cases.  Since our
interest is in the $k(n)^{*}$-module structure and not so much in the
multiplicative structure, we do this for our final result.

\begin{thm}\label{knASS}
The $2$-primary $k(n)^{*}(K_{2})$ as a $k(n)^{*}$-module is the sum of the
following three summands:
$$
P(v_{n})\otimes  \underset{0 < i < n}{\bigotimes} T_{2^{n-i}}(z_{i+1})
$$

\begin{displaymath}
\bigoplus_{j>0}\left( T_{\rho_{n}(j)}(v_{n})\ot P(y_{n,j+1})\ot \Ebar(w_{n,j})
\otimes \underset{0<i \le n}{\bigotimes} E(w_{n,j+i})
 \otimes \underset{s\geq 0 }{\bigotimes}T_{2^{n}}
(z_{n+j+s+1}) \right)
\end{displaymath}

\noindent
\begin{align*}
&\bigoplus_{j>0}
\left(T_{\rho_{n}(j+1/2)}(v_{n})\ot P(y_{n,j+1})
     \otimes \underset{0 < i \le n}{\bigotimes}  E(w_{n,j+i})\right.\\
& \qquad\qquad\qquad\qquad\qquad\left.
     \otimes \overline{TP}_{2^{n}}(z_{n+j+1})
    \otimes \underset{s>0 }{\bigotimes}T_{2^{n}} (z_{n+j+s+1}) \right)
\end{align*}

\noindent
% $$
% \bigoplus_{j>0}
% \left(T_{\rho_{n}(j+1/2)}(v_{n})\ot P(y_{n,j+1})
%      \otimes \underset{0 < i \le n}{\bigotimes}  E(w_{n,j+i})
% \otimes \overline{TP}_{2^{n}}(z_{n+j+1})
%     \otimes \underset{s>0 }{\bigotimes}T_{2^{n}} (z_{n+j+s+1}) \right)
% $$

\noindent plus a computable family of $\zt$'s annihilated by $v_{n}$
coming from the $\E$-free part of $H^{*}K_{2}$.
\end{thm}

\end{document}